\documentclass[final]{siamltex}
\usepackage{amsmath,mathrsfs}
\usepackage{amssymb}
\usepackage{graphicx,subfigure,xspace}
\usepackage{pstricks}
\usepackage[sort,numbers]{natbib}
\usepackage[all]{xy}
\newcommand\subscr[2]{#1_{\textup{#2}}}
\newcommand\upscr[2]{#1^{\textup{#2}}}
\newcommand\upsubscr[3]{#1_{\textup{#2}}^{\textup{#3}}}

\newcommand{\integerspositive}{\mathbb{Z}_{>0}}
\newcommand{\integersnonnegative}{\mathbb{Z}_{\ge 0}}

\newcommand{\rrarrows}{\rightrightarrows}

\newcommand\wbda{\texttt{imbalance-correcting algorithm}\xspace}

\newcommand\wbmda{\texttt{mirror imbalance-correcting
    algorithm}\xspace}

\newcommand\cregular{\texttt{load-pushing algorithm}\xspace}

\newcommand\dslmo{\texttt{imbalance-correcting algorithm with
    self-loop addition}\xspace}

\newcommand{\din}{\subscr{d}{in}}
\newcommand{\wdin}{\upsubscr{d}{in}{w}}
\newcommand{\dout}{\subscr{d}{out}}
\newcommand{\wdout}{\upsubscr{d}{out}{w}}

\newcommand{\ds}{\operatorname{ds}}
\newcommand{\visited}{\operatorname{Vis}}

\newcommand{\Nin}{\upscr{\mathcal{N}}{in}}
\newcommand{\Nout}{\upscr{\mathcal{N}}{out}}

\newcommand{\setmap}[3]{#1:#2 \rrarrows #3}
\newcommand{\until}[1]{\{1,\ldots, #1\}}

\newtheorem{remark}[theorem]{Remark}
\newtheorem{example}[theorem]{Example}

\newcommand{\oprocendsymbol}{\hbox{$\bullet$}}
\newcommand{\oprocend}{\relax\ifmmode\else\unskip\hfill\fi\oprocendsymbol}
\def\eqoprocend{\tag*{$\bullet$}}
\newcommand{\longthmtitle}[1]{\mbox{}\textup{(#1):}}

\newcommand{\Adj}{\operatorname{Adj}}

\renewcommand{\tilde}{\widetilde}

\title{Distributed strategies for generating weight-balanced and
  doubly stochastic digraphs\thanks{This work was supported in part by
    NSF Awards CCF-0917166 and CMMI-0908508.  This manuscript
    substantially improves and extends the preliminary conference
    versions presented as~\cite{BG-JC:09-allerton}
    and~\cite{BG-JC:09-acc}.}}

\author{Bahman Gharesifard and Jorge Cort\'{e}s\thanks{Department of
    Mechanical and Aerospace Engineering, University of California San
    Diego, 9500 Gilman Dr, La Jolla, CA 92093-0411, United States, Ph.
    +1 858-822-7930, Fax. +1 858-822-3107,
    \texttt{\{bgharesifard,cortes\}@ucsd.edu}}}

\begin{document}

\maketitle

\begin{abstract}
  This paper deals with the design and analysis of dynamical systems
  on directed graphs (digraphs) that achieve weight-balanced and
  doubly stochastic assignments.  Weight-balanced and doubly
  stochastic digraphs are two classes of digraphs that play an
  essential role in a variety of coordination problems, including
  formation control, agreement, and distributed optimization.  We
  refer to a digraph as doubly stochasticable (weight-balanceable) if
  it admits a doubly stochastic (weight-balanced) adjacency
  matrix. This paper studies the characterization of both classes of
  digraphs, and introduces distributed dynamical systems to compute
  the appropriate set of weights in each case.  It is known that
  semiconnectedness is a necessary and sufficient condition for a
  digraph to be weight-balanceable. The first main contribution is a
  characterization of doubly-stochasticable digraphs. As a by-product,
  we unveil the connection of this class of digraphs with
  weight-balanceable digraphs.  The second main contribution is the
  synthesis of a distributed strategy running synchronously on a
  directed communication network that allows individual agents to
  balance their in- and out-degrees.  We show that a variation of our
  distributed procedure over the mirror graph has a much smaller time
  complexity than the currently available centralized algorithm based
  on the computation of the graph cycles.
  The final main contribution is the design of two cooperative
  strategies for finding a doubly stochastic weight assignment.  One
  algorithm works under the assumption that individual agents are
  allowed to add self-loops.  For the case when this assumption does
  not hold, we introduce an algorithm distributed over the mirror
  digraph which allows the agents to compute a doubly stochastic
  weight assignment if the digraph is doubly stochasticable and
  announce otherwise if it is not. 
  Various examples illustrate the results.
\end{abstract}



\pagestyle{myheadings}%
\thispagestyle{plain}
\markboth{Bahman Gharesifard and Jorge Cort\'{e}s}{Distributed
  strategies for generating weight-balanced and doubly stochastic
  digraphs}

\section{Introduction}


In the last years there has been considerable interest in
understanding the underpinnings of collective behavior from a
dynamical systems perspective. A variety of phenomena from biology and
physics have been carefully studied, and more are coming into light. A
few examples of a vast literature include oscillator
synchronization~\cite{SHS:00,SHS:03,DJTS:10}, self-organization in
biological
systems~\cite{SG-SAL:93,JKP-SVV-DG:02,SC-JLD-NRF-JS-GT-EB:01}, and
animal grouping and aggregation~\cite{AO:86,SG-SAL-DIR:96}.  In the
study of collective behavior, a major concern is the modeling of the
interactions between individual agents and the understanding on how
local interconnections give rise to global emergent behavior. Such
coordination problems have also become a major research area in
engineering because of the connections with distributed robotics,
networked systems, and autonomy, see
e.g.,~\cite{WR-RWB:08,MM-ME:10,FB-JC-SM:08cor} and references therein.

This paper is a contribution to this buoying field.  From a systems
and controls perspective, one of the main objectives is the design of
dynamical systems that are distributed over a given interaction
topology and whose performance guarantees can be formally established
with regards to the task at hand.  In both analysis and design, the
interaction topology (or graph) of the underlying network is a key
element as it determines the information available to each
individual. Typically, directed interaction topologies (where
interactions among agents are unidirectional) pose greater technical
challenges.

In this paper, we study dynamical systems on two important classes of
directed graphs, weight-balanced and doubly stochastic digraphs.
A digraph is weight-balanced if, at each vertex, the sum of the
weights of the incoming edges is equal to the sum of the weights of
the outgoing edges. A digraph is doubly stochastic if it is
weight-balanced and these sums at each vertex are equal to one.  The
notion of weight-balanced digraph is key in establishing convergence
results of distributed algorithms for
average-consensus~\cite{ROS-JAF-RMM:07,ROS-RMM:03c} and consensus on
general functions~\cite{JC:08-auto} via Lyapunov stability analysis.
Weight-balanced digraphs also appear in the design of leader-follower
strategies under time delays~\cite{JH-YH:07}, virtual leader
strategies under asymmetric interactions~\cite{HS-LW-TC:06} and stable
flocking algorithms for agents with significant inertial
effects~\cite{DL-MWS:07}.
In~\cite{LHT:70}, a traffic-flow problem is introduced with $n$
junction and $m$ one-way streets with the goal of ensuring a smooth
traffic flow. It is shown that the problem can be reduced to computing
weights on the edges of the associated digraph so that it is
weight-balanced.
Furthermore, necessary and sufficient conditions are given for a
digraph to be weight-balanced and a centralized algorithm is presented
for computing the weight on each edge.
Doubly stochastic digraphs also play a key role in networked control
problems. Examples include distributed
averaging~\cite{FB-JC-SM:08cor,ROS-RMM:03c,WR-RWB:08,LX-SB:04} and
distributed convex
optimization~\cite{BJ-MR-MJ:09,AN-AO:09,MZ-SM:09}. Convergence in
gossip algorithms also relies on the structure of doubly stochastic
digraphs, see~\cite{SB-AG-BP-DS:06,JL-ASM-BDOA-CY:09}.

Because of the numerous algorithms available in the literature that
use weight-balanced and doubly stochastic interaction topologies, it
is important to develop distributed strategies that allow individual
agents to find the appropriate weight assignments, i.e., to balance
their in- and out-degrees, so that the overall interaction digraph is
weight-balanced or doubly stochastic.  In particular, we are
interested in designing discrete-time dynamical systems that can be
run on the directed communication network which, in finite time,
converge to a weight-balanced/doubly stochastic assignment.  As a
necessary step towards this goal, it is an important research question
to characterize when a digraph can be given an edge weight assignment
that makes it belong to either category. Our focus in this paper is on
nonzero weight assignments. In addition to its theoretical interest,
the consideration of nonzero weight assignments is also relevant from
a practical perspective, as the use of the maximum number of edges
leads to higher algebraic connectivity~\cite{CWW:05}, which in turn
affects positively the rate of
convergence~\cite{SB-AG-BP-DS:06,JL-ASM-BDOA-CY:09,MC-CWW:07,AG-SB:06}
of the algorithms typically executed over doubly stochastic digraphs.
Alternative versions of the problem, where some edge weights are
allowed to be zero, are also worth exploring, although we do not
consider them here. From a distributed perspective, such problems pose
nontrivial challenges, as individual agents would need a procedure to
determine if severing a particular edge (i.e., setting its weight to
zero) or set of them disconnects the overall digraph. Such procedures
would necessarily require information beyond the immediate
neighborhood of the agents.

Our contributions are threefold.  First, we obtain a characterization
of doubly stochasticable digraphs (a characterization of
weight-balanceable digraphs is already available in the literature,
cf.~\cite{LHT:70}).  As a by-product of our study, we demonstrate that
the set of doubly stochasticable digraphs can be generated by a
special subset of weight-balanced digraphs.  Second, we develop two
discrete-time set-valued dynamical systems on the directed
communication network which converge to a weight-balanced digraph in
finite time. The \wbda is a synchronized distributed strategy on a
digraph in which each agent provably balances her in- and out-degrees
in finite time. In this algorithm, each individual agent can send a
message to one of its out-neighbors and receive a message from her
in-neighbors.
The \wbmda is a provably correct distributed strategy over the mirror
digraph whose time complexity is much smaller than that of the
existing centralized strategy of~\cite{LHT:70} based on the
computation of all cycles of the digraph. Third, we synthesize
discrete-time dynamical systems to construct doubly stochastic
adjacency matrices. The \dslmo achieves this task under the assumption
that individual agents can add self-loops to the structure of the
digraph.  If this is not allowed, we introduce the \cregular, which is
a strategy distributed over the mirror digraph that allows agents to:
(i)~identify if their digraph is doubly stochasticable and, if this is
the case, (ii)~find a set of weights that makes the digraph doubly
stochastic. The algorithm relies on the notion of DS-character of a
strongly connected doubly stochasticable digraph and its design and
correctness analysis draw substantially on distributed solutions to
the maximum flow problem~\cite{RKA-TLM-JBO:93,TLP-IL-MB-SHD:05}.


This paper is organized as follows.
Section~\ref{section:mathematical_p} presents some mathematical
preliminaries.  Section~\ref{section:doubly_stochastic_sec} gives the
necessary and sufficient conditions for the existence of a doubly
stochastic adjacency matrix assignment for a given digraph.
Section~\ref{section:wb_algos} introduces two weight-balancing
algorithms that allow each agent to balance her in- and
out-degrees. We characterize their distributed character as well as
the convergence and complexity properties.  In 
Section~\ref{section:ds_dist}, we discuss the problem of designing
cooperative strategies that allow agents to find a doubly stochastic
edge weight assignment.  Finally, Section~\ref{sec:conclusion_sec}
contains our conclusions and ideas for future work.

\section{Mathematical preliminaries} \label{section:mathematical_p}

We adopt some basic notions from~\cite{FB-JC-SM:08cor,RD:05,NB:94}.  A
\emph{directed graph}, or simply \emph{digraph}, is a pair $G=(V,E)$,
where $V$ is a finite set called the vertex set and $ E \subseteq
V\times V $ is the edge set.  If $|V|=n$, i.e., the cardinality of $V$
is $ n\in \mathbb{Z}_{>0}$, we say that $G$ is of order $n$ which,
unless otherwise noted, is the standard assumption throughout the
paper.  We say that an edge $ (u,v) \in E $ is \emph{incident away
  from} $u$ (or an \emph{out-edge} of $ u $) and \emph{incident
  toward} $v$ (or an \emph{in-edge} of $ v $), and we call $ u $ an
\emph{in-neighbor} of $ v $ and $ v $ an \emph{out-neighbor} of $ u $.
We denote the set of in-neighbors and out-neighbors of $ v $,
respectively, with $ \upscr{\mathcal{N}}{in}(v) $ and $
\upscr{\mathcal{N}}{out}(v) $. The \emph{in-degree} and
\emph{out-degree} of $v$, denoted $\din(v)$ and $\dout(v)$, are the
number of in-neighbors and out-neighbors of $v$, respectively.  We
call a vertex $ v $ \emph{isolated} if it has zero in- and
out-degrees.  An \emph{undirected graph}, or simply \emph{graph}, is a
pair $G=(V,E)$, where $V$ is a finite set called the vertex set and
the edge set $E$ consists of unordered pairs of vertices. In a graph,
neighboring relationships are always bidirectional, and hence we
simply use neighbor, degree, etc. for the notions introduced above.  A
graph is \emph{regular} if each vertex has the same number of
neighbors. 
The
\emph{union} $G_1\cup G_2 $ of digraphs $ G_1=(V_1,E_1) $ and $
G_2=(V_2,E_2) $ is defined by $G_1\cup G_2=(V_1\cup V_2, E_1\cup
E_2)$.  The intersection of two digraphs can be defined similarly.  A
digraph $G$ is \emph{generated} by a set of digraphs $G_1, \dots, G_m$
if $ G=G_1\cup \dots \cup G_m $.  We let $ E^{-} \subseteq V\times V $
denote the set obtained by changing the order of the elements of $ E
$, i.e., $(v,u)\in E^{-}$ iff $ (u,v) \in E $.  The digraph
$\overline{G}=(V,E \cup E^{-})$ is the \emph{mirror} of $G$.

A \emph{weighted digraph} is a triplet $ G=(V,E,A) $, where $ (V,E) $
is a digraph and $ A \in \mathbb{R}^{n\times n}_{\geq0} $ is the
\emph{adjacency matrix}. We denote the entries of $ A $ by $ a_{ij} $,
where $ i,j\in\{1,\ldots,n \} $. The adjacency matrix has the property
that the entry $ a_{ij}>0 $ if $ (v_i,v_j)\in E $ and $ a_{ij}=0 $,
otherwise. If a matrix $A$ satisfies this property, we say that $ A $
can be assigned to the digraph $ G=(V,E)$.
Note that any digraph can be trivially seen as a weighted digraph by
assigning weight $1$ to each one of its edges. We will find it useful
to extend the definition of union of digraphs to weighted digraphs.
The union $G_1\cup G_2 $ of weighted digraphs $ G_1=(V_1,E_1,A_1) $
and $ G_2=(V_2,E_2,A_2) $ is defined by $G_1\cup G_2=(V_1\cup V_2,
E_1\cup E_2, A)$, where
\[
A|_{V_1\cap V_2}=A_1|_{V_1\cap V_2}+A_2|_{V_1\cap V_2},\quad
A|_{V_1\backslash V_2}=A_1, \quad A|_{V_2\backslash V_1}=A_2.
\]
For a weighted digraph, the weighted out-degree, weighted in-degree
and imbalance of $v_i$, $i \in \{1,\dots,n\}$, are respectively,
\[
\wdout(v_i) =\sum_{j=1}^{n}a_{ij}, \qquad \wdin(v_i)=\sum_{j=1}^n
a_{ji} , \qquad \omega(v_i)=\wdin(v_i)-\wdout(v_i).
\]
It is worth noticing that the imbalances across the graph always satisfy
\begin{align}
  \label{eq:basic-fact}
   \sum_{i=1}^n \omega(v_i) = 0.
\end{align}

\subsection{Graph connectivity notions}

A \emph{directed path} in a digraph, or in short path, is an ordered
sequence of vertices so that any two consecutive vertices in the
sequence are an edge of the digraph. A \emph{cycle} in a digraph is a
directed path that starts and ends at the same vertex and has no other
repeated vertex.  Hence, a self-loop is a cycle while an isolated
vertex is not.  Two cycles are \emph{disjoint} if they do not have any
vertex in common.  We denote by $\subscr{G}{cyc}$ a union of some
disjoint cycles of $G$ (note that $\subscr{G}{cyc}$ can be just one
cycle).  A \emph{semi-cycle} of a digraph is a cycle of its mirror
graph.  Note that a semi-cycle is a directed path in the mirror
digraph.  A digraph is called \emph{acyclic} if it contains no
cycles. A \emph{directed tree} is an acyclic digraph which contains a
vertex called the \emph{root} such that any other vertex of the
digraph can be reached by one and only one directed path starting at
the root. A \emph{breadth-first spanning tree $ \mathrm{BFS}(G,v) $ of
  a digraph $ G=(V,E) $ rooted at $ v \in V $} is a directed tree
rooted at $ v $ that contains a shortest path from $ v $ to every
other vertex of $ G $, see~\cite{FB-JC-SM:08cor, GT:01}.

A digraph is \emph{strongly connected} if there is a path between each
pair of distinct vertices and is \emph{strongly semiconnected} if the
existence of a path from $ v $ to $ w $ implies the existence of a
path from $ w $ to $ v $, for all $ v,w\in V $. Clearly, strong
connectedness implies strong semiconnectedness, but the converse is
not true.  The \emph{strongly connected components} of a directed
graph $ G $ are its maximal strongly connected subdigraphs.

\subsection{Basic notions from linear algebra}
A matrix $ A \in \mathbb{R}^{n\times n}_{\geq 0}$ is
\emph{weight-balanced} if $ \sum_{j=1}^{n}a_{ij}=\sum_{j=1}^{n}a_{ji}
$, for all $ i\in \{1,\ldots,n\} $.  A matrix $ A \in
\mathbb{R}^{n\times n}_{\geq 0}$ is \emph{row-stochastic} if each of
its rows sums 1. One can similarly define a column-stochastic matrix.
We denote the set of all row-stochastic matrices on $ \mathbb{R}_{\geq
  0}^{n\times n} $ by $ \mathrm{RStoc}(\mathbb{R}_{\geq 0}^{n\times
  n}) $.  A non-zero matrix $ A \in \mathbb{R}^{n\times n}_{\geq 0} $
is \emph{doubly stochastic} if it is both row-stochastic and
column-stochastic.  A matrix $ A \in \{0,1\}^{n\times n} $ is a
\emph{permutation matrix}, where $ n \in \mathbb{Z}_{\geq 1} $, if $ A
$ has exactly one entry $ 1 $ in each row and each column. A matrix $
A \in \mathbb{R}^{n\times n}_{\geq 0}$ is \emph{irreducible} if, for
any nontrivial partition $ J\cup K $ of the index set $ \{1, \ldots,
n\}$, there exist $ j \in J $ and $ k\in K $ such that $ a_{jk} \neq
0$. We denote by $ \mathrm{Irr}(\mathbb{R}_{\geq 0}^{n\times n}) $ the
set all irreducible matrices on $ \mathbb{R}_{\geq 0}^{n\times n} $.
Note that a weighted digraph $ G $ is strongly connected if and only
if its adjacency matrix is irreducible~\cite{NB:94}.

One can extend the adjacency matrix associated to a disjoint union of
cycles $ \subscr{G}{cyc}$ of $G$ to a matrix $ \subscr{A}{cyc}\in
\mathbb{R}^{n\times n}$ by adding zero rows and columns for the
vertices of $ G $ that are not included in $ \subscr{G}{cyc} $. Note
that the matrix $ \subscr{A}{cyc} $ is the adjacency matrix for a
subdigraph of $ G $. We call $ \subscr{A}{cyc}$ the \emph{extended
  adjacency matrix} associated to $ \subscr{G}{cyc} $.  The following
result establishes the relationship between cycles of $ G $ and
permutation matrices.
\begin{lemma}[Cycles and permutation matrices]\label{lemma:perm_cycle}
  The extended adjacency matrix associated to a union of disjoint
  cycles $ \subscr{G}{cyc} $ is a permutation matrix if and only if $
  \subscr{G}{cyc} $ contains all the vertices of $ G $.
\end{lemma}
\begin{proof}
  It is clear that if $ \subscr{G}{cyc} $ is a union of some disjoint
  cycles and contains all the vertices of $ G $, then the adjacency
  matrix associated to $ \subscr{G}{cyc} $ is a permutation
  matrix. Conversely, suppose that $ \subscr{G}{cyc} $ does not
  contain one of the vertices of $G$. Then the adjacency matrix
  associated to $ \subscr{G}{cyc} $ has a zero row and thus is not a
  permutation matrix.
\end{proof}

\subsection{Weight-balanced and doubly stochastic digraphs}

A weighted digraph $G$ is \emph{weight-balanced} (resp. \emph{doubly
  stochastic}) if its adjacency matrix is weight-balanced
(resp. doubly stochastic). Note that $ G $ is weight-balanced if and
only if $ \wdout(v) =\wdin(v) $, for all $ v \in V $.
A digraph is called \emph{weight-balanceable} (resp. \emph{doubly
  stochasticable}) if it admits a weight-balanced (resp. doubly
stochastic) adjacency matrix.  The following two results characterize
when a digraph is weight-balanceable.

\begin{theorem}[\cite{LHT:70}]\label{Theorem:Theorem_1_Loh}
  A digraph $ G=(V,E) $ is weight-balanceable if and only if the edge set
  $ E $ can be decomposed into $ k $ subsets $ E_1,\ldots, E_k $ such
  that 
  \begin{enumerate}
  \item $ E=E_1\cup E_2\cup \ldots \cup E_k $ and
  \item each $ G=(V,E_i) $, $ i=\{1,\ldots, k\} $, is
    weight-balanceable.
  \end{enumerate}
\end{theorem}

\begin{theorem}[\cite{LHT:70}]\label{Theorem:Theorem_2_Loh}
  Let $ G =(V,E)$ be a digraph. The following statements are
  equivalent:
  \begin{enumerate}
  \item $ G $ is weight-balanceable,
  \item Every element of $ E $ lies in a cycle,
  \item $ G $ is strongly semiconnected.
  \end{enumerate}
\end{theorem}

The proofs of these theorems are constructive but rely on the
computation of all the cycles of the digraph. The basic idea is that
if one can find all the cycles for a given strongly semiconnected
digraph, then, by taking the weighted union of these cycles (i.e., by
assigning to each edge the number of times it appears in the cycles),
one arrives at a weight-balanced assignment, see~\cite{LHT:70}.

Note that Theorem~\ref{Theorem:Theorem_2_Loh} implies that any
strongly semiconnected digraph can be generated by the cycles
contained in it. Therefore, it makes sense to define a minimal set of
cycles with this property. This motivates the introduction of the
following concept.

\begin{definition}[Principal cycle set]\label{def:independent_cycles}
  Let $ G $ be a strongly semiconnected digraph.  Let $ \mathcal{C}(G)
  $ denote the set of all subdigraphs of $ G $ that are either
  isolated vertices, cycles of $ G $, or a union of disjoint cycles of
  $ G $.  $P \subseteq \mathcal{C}(G) $ is a \emph{principal cycle
    set} of $ G $ if its elements generate $G$, and there is no subset
  of $ \mathcal{C}(G) $ with strictly smaller cardinality that
  satisfies this property.
\end{definition}

Note that there might exist more than one principal cycle
set. However, by definition, the cardinalities of all principal cycle
sets are the same. We denote this cardinality by $p(G)$.  Note that a
cycle, or a union of disjoint cycles, that contains all the vertices
has the maximum number of edges that an element of $\mathcal{C}(G) $
can have.  Thus these elements are the obvious candidates for
constructing a principal cycle set. Principal cycle sets give rise to
weight-balanced assignments, as we state next.

\begin{lemma}[Weight-balancing via principal cycle
  sets]\label{le:H_modified}
  Let $G$ be a strongly semiconnected digraph. Then, the union of the
  elements of a principal cycle set $P$ of $G$, considered as
  subdigraphs with trivial weight assignment, gives a set of positive
  integer weights which make the digraph weight-balanced.
\end{lemma}
\begin{proof}
  Since each element of a principal cycle $P$ is either an isolated
  vertex, a cycle, or union of disjoint cycles, it is
  weight-balanced. By definition, $G$ can be written as the union of
  the elements of $P$.  Thus by Theorem~\ref{Theorem:Theorem_1_Loh},
  the weighted union of the elements of $P$ gives a set of weights
  makes the digraph $G$ weight-balanced.
\end{proof}

In general, the assignment in Lemma~\ref{le:H_modified} uses fewer
number of cycles than the ones used in
Theorem~\ref{Theorem:Theorem_2_Loh}.

\subsection{Discrete set-valued analysis}

Here, we provide a brief exposition of useful concepts from
discrete-time set-valued dynamical systems
following~\cite{DGL:84,FB-JC-SM:08cor}. For $ X \subseteq
\mathbb{R}^n$, let $\setmap{F}{X}{X}$ denote a set-valued map that
takes a point in $X$ to a subset $F(x)$ of $X$. $F$ is non-empty if $
F(x) \neq \emptyset$, for all $ x \in X $.  A point $ x^* \in X $ is a
\emph{fixed point} of $ F $ if $ x^* \in F(x^*) $.  An
\emph{evolution} of $ F $ on $ X $ is any trajectory $ \gamma:
\mathbb{Z}_{\geq0} \rightarrow X $ such that
\begin{align*}
  \gamma(k+1)\in F(\gamma(k)) , \quad \text{for all } k \in
  \mathbb{Z}_{\geq0} .
\end{align*}
The map $ F $ is \emph{closed} at $ x \in X $ if, for any two
convergent sequences $ \{x_k\}_{k=0}^{\infty} $ and
$\{y_k\}_{k=0}^{\infty} $, with $ \lim_{k\rightarrow \infty} x_k =x$,
$ \lim_{k\rightarrow \infty} y_k =y$, and $ y_k \in F(x_k) $, for all
$ k \in \mathbb{Z}_{\geq0} $, we have $ y \in F(x) $.  The map $F$ is
closed on $X$ if it is closed at $x$, for all $x \in X$.
A set $ W \subset X $ is \emph{weakly positively invariant} with
respect to $ F $ if for any $ x \in W $ there exists $ y \in W $ such
that $ y \in F(x) $. $ W $ is \emph{strongly positively invariant}
with respect to $ F $ if $ F(x) \subset W $, for all $ x \in W $.
Finally, a continuous function $ V: X\rightarrow \mathbb{R} $ is
called \emph{non-increasing} along $ F $ in $ W \subset X $ if $
V(y)\leq V(x) $, for all $ x \in W $ and $ y \in F(x) $.  Equipped
with these tools, one can formulate the following set-valued version
of the LaSalle invariance principle, which will be most useful in the
developments later.

\begin{theorem}[LaSalle invariance principle for discrete-time
  set-valued dynamical systems]\label{theorem:LaSalle}
  Let $\setmap{F}{X}{X}$ be a set-valued map on $ X \subset
  \mathbb{R}^n $ and let $ W \subset X $ be a closed and strongly
  positively invariant with respect to $ F $.  Suppose $ F $ is
  non-empty and closed on $ W $ and all evolutions of $ F $ with
  initial condition in $ W $ are bounded. Let $ V: X\rightarrow
  \mathbb{R}$ be continuous and non-increasing function along $ F $ on
  $ W$. Then, any evolution of $F$ with initial condition in $ W $
  approaches a set of the form $ S \cap V^{-1}(c) $, where $ c \in
  \mathbb{R} $ and $ S $ is the largest weakly positively invariant
  set contained in $ \{x \in W \ | \ \mathrm{there} \ \mathrm{exists}
  \ y \in F(x) \ \mathrm{such} \ \mathrm{that} \ V(x)=V(y)\} $.
\end{theorem}

\section{When does a digraph admit a doubly stochastic weight
  assignment?}\label{section:doubly_stochastic_sec}

Our main goal in this section is to obtain necessary and sufficient
conditions that characterize when a digraph is doubly
stochasticable. This characterization is a necessary step before
addressing in later sections the design of distributed dynamical
systems that find doubly stochastic weight assignments. As an
intermediate step of this characterization, we will also find it
useful to study the relationship between weight-balanced and doubly
stochastic digraphs.

Note that strong semiconnectedness is a necessary, and sufficient,
condition for a digraph to be weight-balanceable. All doubly
stochastic digraphs are weight-balanced; thus a necessary condition
for a digraph to be doubly stochasticable is strong
semiconnectedness. Moreover, weight-balanceable digraphs that are
doubly stochasticable do not have any isolated vertex.  However, none
of these conditions is sufficient.  A simple example illustrates
this. Consider the digraph shown in
Figure~\ref{fig:ds_digraph_ex1}. Note that this digraph is strongly
connected; thus there exists a set of positive weights which makes the
digraph weight-balanced.  However, there exists no set of nonzero
weights that makes this digraph doubly stochastic.
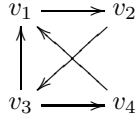
\begin{figure}
  \[
  \xymatrix{v_1 \ar[r] & v_2 \ar[dl]\\
    v_3 \ar[u] \ar[r] & v_4 \ar[lu] }
  \]
  \caption{A weight-balanceable digraph for which their exists no
    doubly stochastic adjacency assignment.}
  \label{fig:ds_digraph_ex1}
\end{figure}
Suppose
\[
A=\left(\begin{array}{cccc}0 & \alpha_1 & 0 &0\\
    0 & 0 & \alpha_2 & 0\\
    \alpha_3 & 0 & 0 & \alpha_4\\
    \alpha_5 & 0 & 0 & 0
  \end{array}\right), 
\]
where $ \alpha_i \in \mathbb{R}_{>0} $, for all $ i\in\{1,\ldots, 5\}
$, is a doubly stochastic adjacency assignment for this digraph. Then
a simple computation shows that the only solution that makes the
digraph doubly stochastic is by choosing $ \alpha_3=0 $, which is not
possible by assumption. Thus this digraph is not doubly
stochasticable. 

The following result will simplify our analysis by allowing us to
restrict our attention to strongly connected digraphs.

\begin{lemma}[Strongly connected components of a doubly stochasticable
  digraph]\label{le:decomp}
  A strongly semiconnected digraph is doubly stochasticable if and
  only if all of its strongly connected components are doubly
  stochasticable.
\end{lemma}

\begin{proof}
  Let $G_1$ and $G_2$ be two strongly connected components of the
  digraph. Note that there can be no edges from $v_1 \in G_1$ to $v_2
  \in G_2$ (or vice versa). If this were the case, then the strong
  semiconnectedness of the digraph would imply that there is a path
  from $v_2$ to $v_1$ in the digraph, and hence $G_1 \cup G_2$ would
  be strongly connected, contradicting the fact that $G_1$ and $G_2$
  are maximal.  Therefore, the adjacency matrix of the digraph is a
  block-diagonal matrix, where each block corresponds to the adjacency
  matrix of a strongly connected component, and the result follows.
\end{proof}

As a result of Lemma~\ref{le:decomp}, we are interested in
characterizing the class of strongly connected digraphs which are
doubly stochasticable.

\subsection{The relationship between weight-balanced and doubly
  stochastic adjacency matrices}\label{section:wb_ds}
  
As an intermediate step of the characterization of doubly
stochasticable digraphs, we will find it useful to study the
relationship between weight-balanced and doubly stochastic digraphs.
The example in Figure~\ref{fig:ds_digraph_ex1} underscores the
importance of characterizing the set of weight-balanceable digraphs
that are also doubly-stochasticable.

We start by introducing the \emph{row-stochastic normalization map} $
\phi: \mathrm{Irr}(\mathbb{R}_{\geq 0}^{n\times n})\rightarrow
\mathrm{RStoc}(\mathbb{R}_{\geq 0}^{n\times n}) $ defined by
\[
\phi: (a_{ij}) \mapsto \left( \frac{a_{ij}}{\sum_{l=1}^na_{il}}
\right) .
\]
Note that, for $ A \in \mathrm{Irr}(\mathbb{R}_{\geq 0}^{n\times n})
$, $\phi (A)$ is doubly stochastic if and only if
\[
\sum_{i=1}^n\frac{a_{ij}}{\sum_{l=1}^na_{il}}=1,
\]
for all $ j\in \{1,\dots, n\} $.  The following result characterizes
when the digraph associated with an irreducible weight-balanced
adjacency matrix is doubly stochasticable.

\begin{theorem}[Weight-balanced and doubly stochasticable
  digraphs]\label{theorem:dsable_wbs}
  Let $ A \in \mathrm{Irr}(\mathbb{R}_{\geq 0}^{n\times n}) $ be an
  adjacency matrix associated to a strongly connected weight-balanced
  digraph.  Then $ \phi(A) $ is doubly stochastic if and only if $
  \sum_{l=1}^na_{il}=C $, for all $ i \in \{1,\ldots, n\} $, for some
  $ C \in \mathbb{R}_{> 0}$.
\end{theorem}
\begin{proof}
  The implication from right to left is immediate. Suppose then that $
  A $ is associated to a strongly connected weight-balanced
  digraph. Then we need to show that if $ A $ satisfies the following
  set of equations
  \begin{align*}
    \sum_{l=1}^na_{jl}=\sum_{l=1}^na_{lj},\qquad
    \sum_{i=1}^n\frac{a_{ij}}{\sum_{l=1}^na_{il}}=1,
  \end{align*}
  for all $ j\in\{1,\ldots, n\} $, there exists $ C \in
  \mathbb{R}_{> 0} $ such that $ \sum_{l=1}^na_{il}=C $, for all $ i
  \in \{1,\ldots, n\} $. Let $ C_{k}=\sum_{l=1}^na_{kl}$,
  $ k\in\{1,\ldots, n\} $. Then the doubly stochastic conditions can
  be written as
  \begin{equation}\label{eq:ds}
    \frac{a_{1j}}{C_1}+\frac{a_{2j}}{C_2}+\dots+\frac{a_{nj}}{C_n}=1,
  \end{equation}
  for all $ j\in \{1,\ldots, n\} $. Note that $ C_k\neq 0 $, for all $
  k\in\{1,\ldots, n\} $, since $ A $ is irreducible. By
  the weight-balanced assumption, we have
  \[
  a_{1j}+a_{2j}+\dots+a_{nj}=C_j,
  \]
  for all $ j\in \{1,\ldots, n\} $. Thus
  \begin{equation}\label{eq:wb}
    \frac{a_{1j}}{C_j}+\frac{a_{2j}}{C_j}+\dots+\frac{a_{nj}}{C_j}=1,
  \end{equation}
  From~\eqref{eq:ds} and~\eqref{eq:wb}, we have
  \begin{equation}\label{eq:main}
    a_{1j}\left(\frac{1}{C_1}-\frac{1}{C_j}\right)+\dots+a_{nj}  
    \left(\frac{1}{C_n}-\frac{1}{C_j}\right)
    =0,
  \end{equation}
  for all $ j\in \{1,\ldots, n\} $.  Suppose that, up to rearranging,
  \[
  C_1=\min_{k}\{C_k \ | \ k\in\{1,\ldots, n\}\},
  \]
  and, $ 0<C_1<C_i $, for all $ i \in\{2,\dots, n\}$. 
  %
  %
  Then~\eqref{eq:main} gives
  \[
  a_{21}\left(\frac{1}{C_2} - \frac{1}{C_1}\right) + \dots +
  a_{n1}\left(\frac{1}{C_n}-\frac{1}{C_1}\right)=0;
  \]
  thus $ a_{j1}=0 $, for all $ j\in\{2,\dots, n\} $, which contradicts
  the irreducibility assumption. If the set $ \{C_k\}_{k=1}^n $ has
  more than one element giving the minimum, the proof follows a
  similar argument.  Suppose
  \[
  C_1=C_2=\min_{k}\{C_k \ | \ k\in\{1,\ldots, n\}\},
  \]
  and suppose that $ 0<C_1 = C_{2}<C_i $, for all $ i \in\{3,\dots,
  n\}$. Then we have
  \[
  a_{31}\left(\frac{1}{C_3} - \frac{1}{C_1}\right) + \dots +
  a_{n1}\left(\frac{1}{C_n}-\frac{1}{C_1}\right)=0,
  \]  
  and
  \[
  a_{32}\left(\frac{1}{C_3} - \frac{1}{C_2}\right) + \dots +
  a_{n2}\left(\frac{1}{C_n}-\frac{1}{C_2}\right)=0,
  \]  
  and thus $a_{j1}=0=a_{j2}$, for all $j \in \{3, \dots,n \}$, which
  contradicts the irreducibility assumption. The same argument holds
  for an arbitrary number of minima.
\end{proof}

\begin{corollary}[Self-loop addition makes a digraph doubly
  stochasticable]\label{corollary:3_adding_loop}
  Any strongly connected digraph is doubly stochasticable after adding
  enough number of self-loops.
\end{corollary}
\begin{proof}
  Any strongly connected digraph is weight-balanceable. The result
  follows from noting that, for any weight-balanced matrix, it is
  enough to add self-loops with appropriate weights to the vertices of
  the digraph to make the conditions of
  Theorem~\ref{theorem:dsable_wbs} hold.
\end{proof}

Regular undirected graphs trivially satisfy the conditions of
Theorem~\ref{theorem:dsable_wbs} and hence the following result.
\begin{corollary}[Undirected regular graphs]
  All undirected regular graphs are doubly stochasticable.
\end{corollary}

We capture the essence of Theorem~\ref{theorem:dsable_wbs}  with
the following definition. 
\begin{definition}[$ C $-regularity]
  Let $ G=(V,E) $ be a strongly connected digraph and let $ A $ be a
  weight-balanced adjacency matrix which satisfies the conditions of
  Theorem~\ref{theorem:dsable_wbs} with $ C\in \mathbb{R}_{>0} $. Then
  we refer to $ G=(V,E,A) $ as a \emph{$ C $-regular digraph}.
\end{definition}

\subsection{Necessary and sufficient conditions for doubly
  stochasticability}\label{sec:ds}

In this section, we provide a characterization of the structure of
digraphs that are doubly stochasticable. We start by giving a
sufficient condition for doubly stochasticability.

\begin{proposition}[Sufficient condition for doubly
  stochasticability]\label{prop:main_1}
  A strongly connected digraph $G$ is doubly stochasticable if there
  exists a set $ \{\subscr{G}{cyc}^i\}_{i=1}^{\xi} \subseteq
  \mathcal{C}(G) $, where $ \xi \geq p(G) $, that generates $G$ and
  such that, for each $ i\in \{1,\ldots, \xi\} $, $ \subscr{G}{cyc}^i$
  contains all the vertices of $G$.
\end{proposition}
\begin{proof} 
  Suppose $ G=\cup_{i=1}^{\xi}\subscr{G}{cyc}^i $,
  where $ \subscr{G}{cyc}^i\in \mathcal{C}(G) $ contain all the
  vertices, for all $ i\in \{1,\ldots, \xi\} $.  Consider the
  adjacency matrix
  \begin{equation*}
    A=\sum_{i=1}^{\xi}\subscr{A}{cyc}^i,
  \end{equation*}
  where $ \subscr{A}{cyc}^i $ is the extended adjacency matrix
  associated to $ \subscr{G}{cyc}^i $. Note that $A$ is
  weight-balanced and satisfies the condition of
  Theorem~\ref{theorem:dsable_wbs} (the sum of each row is equal to
  $\xi$). Thus $G$ is doubly stochasticable.
\end{proof}

Proposition~\ref{prop:main_1} suggests the definition of the following
notion. Given a strongly connected digraph $G$ that satisfies the
sufficient condition of Proposition~\ref{prop:main_1}, $DS(G)\subseteq
\mathcal{C}(G) $ is a \emph{DS-cycle set} of $G $ if all its elements
contain all the vertices of $G$, they generate $G$, and there is no
subset of $ \mathcal{C}(G) $ with strictly smaller cardinality that
satisfies these properties.  The cardinality of any DS-cycle set of $
G $ is the \emph{DS-character} of $ G $, denoted $ \ds(G)$.  By
definition, $ \ds(G)\geq p(G) $, where recall that $p(G)$ denotes the
cardinality of any principal cycle set.

The following result states that the condition of
Proposition~\ref{prop:main_1} is actually necessary for a digraph to be
doubly stochasticable.
  
\begin{proposition}[Necessary condition for doubly
  stochasticability]\label{prop:dsable}
  Let $G$ be a strongly connected digraph. Suppose that one can assign
  a doubly stochastic adjacency matrix $ A$ to $G$. Then $ G $
  satisfies the condition of Proposition~\ref{prop:main_1} and
  \[
  A = \sum_{i=1}^{\xi}\lambda_i\subscr{A}{cyc}^i,
  \]
  where 
  \begin{itemize}
  \item $\{ \lambda_i \}_{i=1}^{\xi} \subset \mathbb{R}_{> 0}$, $
    \sum^{\xi}_{i=1}\lambda_i=1 $, and $ \xi \geq \ds(G) $.
  \item $ \subscr{A}{cyc}^i $, $ i\in\{1,\ldots,\xi\} $, is the
    extended adjacency matrix associated to an element of $
    \mathcal{C}(G) $ that contains all the vertices.
  \end{itemize}
\end{proposition}
\begin{proof}
  Let $A $ be a doubly stochastic matrix associated to $ G $. By the
  Birkhoff\textendash von Neumann theorem~\cite{GB:1946}, a square
  matrix is doubly stochastic if and only if it is a convex
  combination of permutation matrices. Therefore,
  \[
  A =\sum_{i=1}^{n!}\bar{\lambda}_i\subscr{A}{perm}^i,
  \]
  where $ \bar{\lambda}_i\in \mathbb{R}_{\geq0} $, $
  \sum_{i=1}^{n!}\bar{\lambda}_i=1 $, and $ \subscr{A}{perm}^i $ is a
  permutation matrix for each $ i \in\{1,\ldots, n!\} $.  By
  Lemma~\ref{lemma:perm_cycle}, for all $ \bar{\lambda}_i> 0 $, one
  can associate to the corresponding $ \subscr{A}{perm}^i $ a union of
  disjoint cycles that contains all the vertices. Thus each
  $\subscr{A}{perm}^i $ is an extended adjacency matrix associated to
  an element of $ \mathcal{C}(G) $ that contains all the vertices.
  This proves that there exists a set $
  \{\subscr{G}{cyc}^i\}_{i=1}^{\xi} \subseteq \mathcal{C}(G) $, where
  $ \xi \geq p(G) $, that generates $G$; thus $ G $ satisfies the
  condition of Proposition~\ref{prop:main_1}.  Let us rename all the
  nonzero coefficients $ \bar{\lambda}_i>0 $ to $ \lambda_i $. In
  order to complete the proof, we need to show that at least $ \ds(G)$
  of the $ \lambda_i $'s are nonzero. Suppose otherwise. Since each $
  \subscr{A}{perm}^i $ with nonzero coefficient is associated to an
  element of $ \mathcal{C}(G) $, this means that the digraph $ G $ can
  be generated by fewer elements than $ \ds(G) $, which would contradict
 the definition of DS-character.
\end{proof}

Note that Propositions~\ref{prop:main_1} and~\ref{prop:dsable} fully
characterize the set of doubly stochasticable strongly connected
digraphs. We gather this result in the following corollary.
\begin{corollary}[Necessary and sufficient condition for doubly
  stochasticability]\label{corollary:main}
  A strongly connected digraph $G$ is doubly stochasticable if and
  only if there exists a set $ \{\subscr{G}{cyc}^i\}_{i=1}^{\xi}
  \subseteq \mathcal{C}(G) $, where $ \xi \geq \ds(G) $, that generates
  $G$ and such that $ \subscr{G}{cyc}^i $ contains all the vertices of
  $ G $, for each $ i\in \{1,\ldots, \xi\} $.
\end{corollary}

If a doubly stochasticable digraph is not strongly connected, one can
use these results for Corollary~\ref{corollary:main} on each strongly
connected component.

\begin{example}[Weight-balanceable, not doubly stochasticable
    digraph]\label{ex:1} {\rm Consider the digraph $ G $ shown in
    Figure~\ref{fig:ex1}(a). It is shown in~\cite{BG-JC:09-allerton}
    that there exists a set of weights which makes this digraph
    weight-balanced.  We show that the digraph is not doubly
    stochasticable.
    \begin{figure}
      \[(a) \;
      \xymatrix{ v_2 \ar[dr] \ar[ddr] & & v_1 \ar[ll] \\
        & v_3 \ar[ur] &\\
        v_5 \ar[ur] & v_4 \ar[u] \ar[uur] \ar[l] & }\quad (b) \;
      \xymatrix{ v_1 \ar[rr] & & v_2 \ar[d] \\
      v_5 \ar[u] \ar@/^/[dr] & &v_3 \ar[ll] \ar[dl]\\
      & v_4 \ar[ul] \ar[uul]& }
    \]
    \caption{The digraph of Examples~\ref{ex:1} and~\ref{ex:2} are
      shown in plots (a) and (b), respectively.}
      \label{fig:ex1}
    \end{figure}
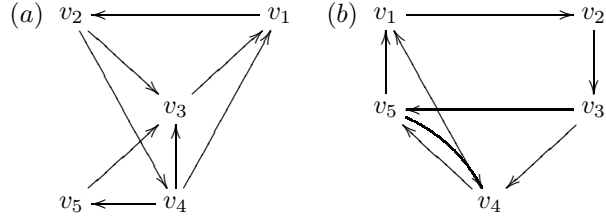
    Suppose there exists a union of disjoint cycles containing all the
    vertices.  Since the edge $ (v_2,v_3) $ only appears in the cycle
    $ \subscr{G}{cyc}=\{v_1,v_2,v_3\} $, one element in such a union
    must contain this cycle. But then it is impossible for this
    element to also contain $ v_4 $ and $ v_5 $, as every cycle
    containing $ v_4 $ and $ v_5 $ also contains at least one of the
    vertices $ \{v_1,v_2,v_3\} $.  Thus by
    Corollary~\ref{corollary:main}, there exists no doubly stochastic
    adjacency assignment for this digraph. One can verify this by
    trying to find such assignment explicitly, i.e., by seeking $
    \alpha_i \in \mathbb{R}_{>0} $, where $ i\in\{1,\ldots,8\} $, such
    that
    \[
    A=\left(\begin{array}{ccccc}0 & \alpha_1 & 0 &0 & 0\\
        0 & 0 & \alpha_2 & \alpha_3 & 0\\
        \alpha_4 & 0 &0 &0 &0\\
        \alpha_5 &0 & \alpha_6 & 0 & \alpha_7\\
        0& 0 &\alpha_8 & 0 & 0
      \end{array}\right)
    \]
    is doubly stochastic. A simple computation shows that such an
    assignment is not possible unless $ \alpha_2=\alpha_5=\alpha_6=0$,
    which is a contradiction.  \oprocend}
\end{example}

\begin{example}[Doubly stochasticable digraph]\label{ex:2}
  {\rm Consider the digraph $ G $ shown in Figure~\ref{fig:ex1}(b).
    One can observe that the only principal cycle set of $ G $
    contains the two cycles shown in Figure~\ref{fig:ex2_pc}.
  \begin{figure}
    \[
    \xymatrix{ v_1 \ar[rr] & & v_2 \ar[d] \\
      v_5 \ar[u]  & &v_3  \ar[dl]\\
      & v_4 \ar[ul] & } \qquad
      \xymatrix{ v_1 \ar[rr] & & v_2 \ar[d] \\
      v_5 \ar@/^/[dr] & &v_3 \ar[ll] \\
      & v_4 \ar[uul]& }
    \]
    \caption{The only principal cycle set for the digraph of
      Example~\ref{ex:2} contains the above cycles.}
    \label{fig:ex2_pc}
  \end{figure}
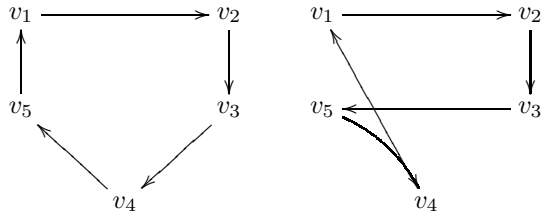
  Both of these cycles pass through all the vertices of the digraph
  and thus, using Corollary~\ref{corollary:main}, this digraph is
  doubly stochasticable. Note that this digraph has another three
  cycles, shown in Figure~\ref{fig:ex2_extra_cycles}, none of which is
  in the principal cycle set.
  \begin{figure}
    \[
    \xymatrix{ v_1 \ar[rr] & & v_2 \ar[d] \\
      & &v_3 \ar[dl]\\
      & v_4 \ar[uul]& }\qquad
      \xymatrix{
        \\
      v_5 \ar@/^/[dr] & &\\
      & v_4 \ar[ul] & }
    \quad
    \xymatrix{ v_1 \ar[rr] & & v_2 \ar[d] \\
      v_5 \ar[u] & &v_3 \ar[ll] }
    \]
    \caption{Cycles of the digraph $ G $ of Example~\ref{ex:2} which
      are not in the principal cycle set.}
    \label{fig:ex2_extra_cycles}
  \end{figure}
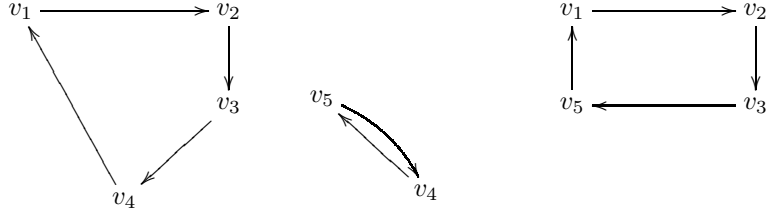
  The adjacency matrix assignment
  \[
  A=\left(\begin{array}{ccccc}
      0 & 2 & 0 &0 & 0\\
      0 & 0 & 2 & 0 & 0\\
      0 & 0 & 0 & 1 & 1\\
      1 & 0 & 0 & 0 & 1\\
      1 & 0 & 0 & 1 &0
    \end{array}\right)
  \]
  obtained by the sum of the elements of the principal cycle set, is
  weight-balanced and satisfies the conditions of
  Theorem~\ref{theorem:dsable_wbs} and thus is doubly stochasticable.
  Also note that not all the weight-balanced adjacency assignments
  become doubly stochastic under the row-stochastic normalization
  map. An example is given by
  \[
  A=\left(\begin{array}{ccccc}
      0 & 3 & 0 &0 & 0\\
      0 & 0 & 3 & 0 & 0\\
      0 & 0 & 0 & 2 & 1\\
      2 & 0 & 0 & 0 & 2\\
      1 & 0 & 0 & 2 &0
    \end{array}\right) .   \eqoprocend
  \]
}
\end{example}
\smallskip


An alternative question to the one considered above would be to find a
set of edge weights (some possibly zero) that make the digraph doubly
stochastic. Such assignments exist for the digraph in
Figure~\ref{fig:ds_digraph_ex1}.  However, such weight assignments are
not guaranteed, in general, to preserve the connectivity of the
digraph.  The following result gives a sufficient condition for the
existence of such a weight assignment.

\begin{proposition}[Doubly stochasticable digraphs via weight
  assignments with some zero entries]\label{prop:zeros-in-weights}
  A strongly connected digraph $ G$ admits an edge weight assignment
  (where some entries might be zero) such that the resulting weighted
  digraph is strongly connected and doubly stochastic if there exists
  a cycle containing all the vertices of~$G$.
\end{proposition}

It is an interesting research question to determine sufficient and
necessary conditions that characterize digraphs that are doubly
stochasticable via weight assignments that can have some zero entries
and still preserve the graph connectivity.  Regarding
Proposition~\ref{prop:zeros-in-weights}, note that even if
connectivity is preserved, fewer edges lead to smaller algebraic
connectivity~\cite{CWW:05}, which in turn affects negatively the rate
of convergence of the consensus, optimization, and gossip algorithms
executed over doubly stochastic digraphs, see
e.g.,~\cite{SB-AG-BP-DS:06,JL-ASM-BDOA-CY:09,MC-CWW:07,AG-SB:06}.

\subsection{Properties of the topological character of doubly
  stochasticable digraphs}\label{section:top_char}
In this section, we investigate the properties of DS-cycle sets and of
their cardinality $ \ds(G) $. First, we show that rational weight
assignments are sufficient to make an adjacency matrix doubly
stochastic.



\begin{lemma}[Rational weight assignments]\label{le:real_rational_ds}
  Assume there exists a real weight assignment that makes a digraph
  $G$ doubly stochastic. Then, there also exists a rational weight
  assignment that makes $ G $ doubly stochastic.
\end{lemma}

\begin{proof}
  By assumption, $ G $ is doubly stochasticable. Thus, by
  Corollary~\ref{corollary:main}, it has a DS-cycle set. Using
  elements of this DS-cycle set, one can construct an integer weight
  assignment $ \subscr{A}{wb} $ that makes $ G $ weight balanced and
  satisfies the conditions of Theorem~\ref{theorem:dsable_wbs} with
  some $ C\in \mathbb{Z}_{>0} $.  Therefore, $
  \frac{1}{C}\subscr{A}{wb} $ gives rise to a rational weight
  assignment that makes $ G $ doubly stochastic.
\end{proof}

Lemma~\ref{le:real_rational_ds} allows us to focus the attention,
without loss of generality, on doubly stochastic adjacency matrices
with rational entries or, alternatively, on weight-balanced adjacency
matrices with integer entries whose rows and columns all sum up to the
same integer. This is what we do in the rest of the paper.

\begin{proposition}[Properties of the DS-character]\label{prop:C_char}
  Let $ G $ be a strongly connected doubly stochasticable digraph of
  order $ n\in \mathbb{Z}_{>0} $ with DS-character $ \ds(G) $. Then
  for $ C\geq \ds(G) $, there exists a weight assignment $
  \subscr{A}{wb} \in \mathbb{Z}_{\geq 0}^{n\times n} $ that makes $
  G $ a $ C $-regular digraph.
\end{proposition}
\begin{proof}
  By Corollary~\ref{corollary:main}, it is clear that one can generate
  $ \subscr{A}{wb} \in \mathbb{Z}_{\geq 0}^{n\times n} $ that makes $
  G $ weight-balanced and also satisfies the conditions of
  Theorem~\ref{theorem:dsable_wbs} for $ C=\ds(G) $, just by taking
  the weighted union of the members of a DS-cycle set.  Let $ C>\ds(G)
  $.  Choose a set of integer numbers $ \lambda_i \in \mathbb{Z}_{> 0}
  $, for $ i \in \{1,\ldots, \ds(G)\} $, such that $
  \sum_{i=1}^{\ds(G)}\lambda_i=C $.  Consider the adjacency matrix
  \[
  A=\sum_{i=1}^{\ds(G)}\lambda_i\subscr{A}{cyc}^i,
  \]
  where $ \subscr{A}{cyc}^i $ is the extended adjacency matrix
  associated to the $ i $th element of the DS-cycle set. The
  matrix $A$ is weight-balanced adjacency and satisfies the conditions
  of Theorem~\ref{theorem:dsable_wbs}.  
\end{proof}


We finish this section by bounding the
DS-character of a digraph.

\begin{lemma}[Bounds for the
  DS-character]\label{lemma:DS_character_bound}
  Let $ G=(V,E) $ be a doubly stochasticable strongly connected
  digraph. Then
  \[
  \max \{\max_{v \in V} \dout(v), \max_{v \in V}
\din(v)\} \leq \ds(G)\leq |E|-|V|+1.
  \]
\end{lemma}\label{lemma:Cstar_bound}
\begin{proof}
  The first inequality follows from the fact that none of the
  out-edges of the vertex $ v $ (similarly, none of of in-edges of
  this vertex)
  are contained in the same element of any DS-cycle set $ DS(G)$. To
  show the second inequality, take any element of $ DS(G)$. This
  element must contain $ |V|$ edges.  The rest of the edges of the
  digraph can be represented by at most $ |E|-|V| $ elements of $
  \mathcal{C}(G)$, and hence the bound follows.  \quad
\end{proof}


\section{Strategies for making a digraph
  weight-balanced}\label{section:wb_algos}

The existing centralized algorithm for constructing a weight-balanced
digraph, proposed in~\cite{LHT:70}, relies on computing all the cycles
of the digraph and thus it is computationally complex.  In this
section, we instead introduce two distributed strategies that are
guaranteed to find a weight-balanced adjacency matrix for a
weight-balanceable digraph and compare their convergence properties.
Given the characterization of Theorem~\ref{Theorem:Theorem_2_Loh}, we
focus on strongly semiconnected digraphs. Since each strongly
connected component of the digraph is completely independent from the
others and can be balanced separately, without loss of generality we
deal with strongly connected digraphs throughout the section.

\subsection{The \wbda}\label{subsection:wb_dis_algo}

Given a strongly connected digraph $ G $, we introduce an algorithm,
distributed over $G$, which allows the agents to balance their in- and
out-degrees. We start by an informal description of the \wbda:
\begin{enumerate}
\item each agent can send messages to her out-neighbors and receive
  messages from her in-neighbors. There is an initial round when each
  agent assigns a weight to each out-edge and sends it to her
  corresponding out-neighbor. In this way, everybody can compute her
  in-degree. After this, at most, only one out-edge per agent is
  changed in each round;

\item for each agent, if the in-degree is more than the out-degree,
  the agent changes the weight on one of the out-edges with the
  minimum weight such that she is balanced and sends a message to the
  corresponding out-neighbor informing her of the change;

\item after receiving the messages, each agent updates the in- and
  out- degrees in the next round and repeats the above process.
\end{enumerate}

Note that this algorithm updates the weights synchronously. In the
following, we give a formal description. Suppose that a communication
network is given by a strongly connected digraph $ G=(V,E)$.  Let $
\Adj(G)\subset \mathbb{R}_{\geq 0}^{n\times n} $ be the subset of all
possible adjacency matrices with nonnegative entries associated to
$G$.  For $ A \in \Adj(G)$ and $ i\in\{1,\ldots,n\} $, let
\begin{align*}
  a_i^* &= \min_{k \in \{1,\dots,n\} \setminus \{i\}}\{a_{ik}\ | \
  a_{ik}\neq 0 \},
  \\
  J_i^* &= \{j\in\{1,\ldots,n\}\setminus \{i\} \ | \ a_{ij}=a_i^*\}.
\end{align*}
The set-valued evolution map $ \subscr{f}{imcor}: \Adj(G )
\rightrightarrows \Adj(G) $ assigns to $A = (a_{ij}) \in \Adj(G)$ the
set
\begin{multline}\label{eq:f}
  \subscr{f}{imcor}(A) = \left\{ B\in \Adj(G) \ | \ \text{for each} \;
    i \in \{1,\ldots, n\}, \text{there exists } j_i^*\in J_i^* \;
    \mathrm{with} \; \right.
  \\
  \left.  b_{ij}=
    \left\{
      \begin{array}{ll}
      a_{ij}+\omega(v_i)& \text{if }  \omega(v_i)>0 \;
      \mathrm{and} \; j=j_i^*
      \\
      a_{ij}, & \mathrm{otherwise}
      \end{array}
    \right. \right\} .
\end{multline}

Note that the discrete-time dynamical system on $\Adj(G)$ defined by
the map $ \subscr{f}{imcor}$ corresponds to the \wbda presented
above. Our strategy to establish the correctness of the algorithm is
then based on characterizing the properties of the set-valued map $
\subscr{f}{imcor}$ and then applying the LaSalle invariance principle,
cf. Theorem~\ref{theorem:LaSalle}. We first establish that the map is
closed.

\begin{lemma}[Closedness of $ \subscr{f}{imcor}
  $]\label{lemma:f_closed}
  The map $ \subscr{f}{imcor} $ is closed on $ \Adj(G) $.
\end{lemma}

\begin{proof}
  Let $ D \in \Adj(G) $ and consider two sequences
  $\{D_k\}_{k=1}^{\infty} \subset \Adj(G)$ and $\{C_k\}_{k=1}^{\infty}
  \subset \Adj(G) $ such that $ C_k\in \subscr{f}{imcor}(D_k) $, for
  all $ k \in \mathbb{Z}_{>0}$, $ \lim_{k\rightarrow \infty}D_k=D $
  and $ \lim_{k\rightarrow \infty}C_k=C $.  Since $ J_i^*(D_k)
  \subseteq \{1,\ldots, n\} $ and $n$ is finite, there must exist a
  set $J_1 \times \dots \times J_n$ in $\until{n}^n$ that appears
  infinitely often in the sequence $\{ J_1^*(D_k) \times \dots \times
  J_n^*(D_k) \}_{k=1}^{\infty}$. Therefore, there exists a subsequence
  of $\{D_k\}_{k=1}^{\infty}$ which, with a slight abuse of notation
  and for simplicity, we denote in the same way, such that $
  J_i^*(D_k)=J_i $, for all $ k \in \mathbb{Z}_{\geq 1} $.  Now,
  because $ C_k\in \subscr{f}{imcor}(D_k) $, there exist
  $(j_1(k),\ldots, j_n(k)) \in J_1\times \ldots \times J_n$ for each $k \in
  \mathbb{Z}_{>0}$ such that one can write
  \begin{align}\label{eq:aux}
    (C_k)_{ij}=
    \begin{cases}
      (D_k)_{ij}+\omega_k(v_i)& \mathrm{if} \quad \omega_k(v_i)>0, \;
      \mathrm{and} \; j=j_i(k) ,
      \\
      (D_k)_{ij}, & \mathrm{otherwise} ,
    \end{cases}
  \end{align}
  where $ (D_k)_{ij} $ and $ (C_k)_{ij} $ denote, respectively, the
  entries of $ D_k $ and $ C_k $, and $\omega_k(v_i)$ is the imbalance
  of vertex $v_i$ in the weight assignment $D_k$. Reasoning as before,
  since $ J_1\times \ldots \times J_n $ is finite, there exists an
  element $ (j_1,\ldots, j_n)$ of this set that appears infinitely
  often in the sequence $\{(j_1(k),\ldots,
  j_n(k))\}_{k=1}^{\infty}$. Therefore, there exists a subsequence,
  which with a slight abuse of notation we denote in the same way,
  such that $(j_1(k),\ldots, j_n(k)) = (j_1,\dots,j_n)$ for all $k \in
  \mathbb{Z}_{>0} $.  Note that, for each $i \in \until{n}$ such that
  $\omega(v_i) >0$, the element $j_i$ is unique since otherwise the
  sequence $ \{C_k\}_{k=1}^{\infty} $ would not be
  convergent. Combining these facts with~\eqref{eq:aux} and taking the
  limit as $ k \rightarrow \infty$, we conclude that $ C \in
  \subscr{f}{imcor}(D) $, and hence $ \subscr{f}{imcor} $ is closed at
  $ D $, as claimed.
\end{proof}

Next, we characterize the fixed points of $\subscr{f}{imcor}$.

\begin{lemma}[Fixed points of
  $\subscr{f}{imcor}$]\label{lemma:eq_f_wb}
  $\subscr{f}{imcor}$ has at least one fixed point. Furthermore, $ A^*
  \in \Adj(G) $ is a fixed point if and only if $A^*$ is
  weight-balanced.
\end{lemma}
\begin{proof}
  Given that the digraph is strongly connected,
  Theorem~\ref{Theorem:Theorem_2_Loh} guarantees that it is
  weight-balanceable, and therefore at least one fixed point exists
  (if $ A\in \Adj(G) $ is weight-balanced, then $\subscr{f}{imcor}(A)
  = \{A\}$, and hence $A$ is a fixed point of $
  \subscr{f}{imcor}$). On the other hand, suppose $ A^*\in \Adj(G) $
  satisfies $A^* \in \subscr{f}{imcor}(A^*) $. We reason by
  contradiction. If $ A^* $ is not weight-balanced,
  then~\eqref{eq:basic-fact} would imply that there exists at least a
  vertex $ v\in V $ with $ \omega(v)>0 $.  From~\eqref{eq:f}, this
  would imply $A^* \notin \subscr{f}{imcor}(A^*)$, which is a
  contradiction.
\end{proof} 

The logic used by the \wbda to update edge weights consists of using
edges with the minimum weight. The next result shows that such logic
is powerful in terms of propagating a token (in our case, an
imbalance) across the network.

\begin{lemma}[Propagation of tokens via out-edges with minimum
  weight]\label{lemma:min_weight}
  Let~$G$ be a strongly connected weighted digraph and consider a
  finite number of tokens initially located at some nodes. If each
  node that possesses a token repeatedly passes it to one of her
  out-neighbors via an out-edge with the minimum weight and adds a
  positive constant to this weight, all the nodes will be visited by
  at least one token after a finite number of iterations.
\end{lemma}
\begin{proof}
  Let $V$ denote the vertex set of $G $.  Let $t_0$ denote an
  arbitrary initial time and $\visited(t,t_0)$ be the set of nodes
  visited by any of the tokens up to time~$t$. Since $V$ is finite and
  $\visited(t,t_0) \subset \visited (t+1,t_0)$, for $t \in
  \integersnonnegative$, we deduce that there exists $T$ such that
  $\visited(t,t_0) = Y(t_0) \subseteq V$, for all $t \ge T$. Let us
  show that $ Y(t_0) = V $ for all $t_0$.  We do this by discarding
  any other possibility. Clearly, $|Y(t_0)|$ cannot be $1$ because if
  a node has a token, she will pass it to some other node.  Let $m <
  |V|$ and assume that we have shown that $ |Y(t_0)| \neq 1,\ldots,
  m-1 $ for all $t_0$. Choose any $t_0$ and let us show that $
  |Y(t_0)|\neq m$. Suppose otherwise, i.e., $|Y(t_0)|=m$.  Since the
  digraph is strongly connected, there exists at least one edge from a
  member of $Y(t_0)$, say $ v_k $, to a member of $ V \setminus
  Y(t_0)$, say $ v_{k+1} $.  Since $v_k \in Y(t_0)$, $v_k$ has one of
  the tokens at some point in time, say $t_1$.  Suppose $ v_k $ sends
  this token via an out-edge to a member of $ Y(t_0) $ and increases
  the weight on this edge; otherwise, $|Y(t_0)|\ge m+1$, which is a
  contradiction. By hypothesis, the tokens never reach $ V\setminus
  Y(t_0) $ and get passed in $Y(t_0)$ while increasing the weight of
  edges among nodes in $Y(t_0)$.  We claim that, after a finite time,
  at least one of the tokens comes back to $ v_k $.  Suppose
  otherwise, i.e., for $t>t_1$, no token will ever visit $v_k$. Since
  $Y(t_1+1) \subseteq Y(t_0)$ and $v_k \not \in Y(t_1+1) $, we
  conclude that $|Y(t_1+1)|\le m-1$, which is a
  contradiction. Therefore, a token must visit $v_k$ after $t_1$. When
  this occurs, node $ v_k $ will choose the out-edge with minimum
  weight for sending the token. The whole process gets repeated and
  thus, after a finite time, the weight on the out-edge to $ v_{k+1} $
  has the minimum weight, and thus $v_k$ sends a token to $v_{k+1}$,
  implying $|Y(t_0)| \ge m+1$, which is a contradiction. Therefore,
  $|Y(t_0)|=|V|$, as claimed.
\end{proof}

The last ingredient we need in order to characterize the convergence
of the \wbda is the Lyapunov function $ \subscr{V}{wb} : \Adj(G)
\rightarrow \mathbb{R}_{\geq 0} $ defined by
\begin{equation}\label{eq:V}
  \subscr{V}{wb}(A)=\sum_{i=1}^n
  |\sum_{j=1}^na_{ij}-\sum_{j=1}^na_{ji}| = \sum_{i=1}^n |\omega(v_i)|.
\end{equation}
This function is continuous on $ \Adj(G) $. Note that $V(A) = 0$ if
and only if $A$ is weight-balanced.  We are now ready to state our
convergence result.

\begin{theorem}[Convergence of the \wbda]\label{theorem:main}
  For a strongly connected digraph $ G $, any evolution under the
  \wbda converges in finite time to a weight-balanced adjacency
  matrix.
\end{theorem}
\begin{proof}
  Recall that the evolutions under the \wbda correspond to
  trajectories of the discrete-time dynamical system defined by
  $\subscr{f}{imcor})$. For an initial condition $A \in \Adj (G)$,
  consider the sublevel set $\subscr{V}{wb}^{-1}(\le
  \subscr{V}{wb}(A)) = \{B\in \Adj(G) \ | \ 0\leq
  \subscr{V}{wb}(B)\leq \subscr{V}{wb}(A) \} $. Since $ \subscr{V}{wb}
  $ is continuous, this set is closed.  Let us show that the
  continuous function $\subscr{V}{wb}$ is non-increasing along $
  \subscr{f}{imcor} $, and consequently $\subscr{V}{wb}^{-1}(\le
  \subscr{V}{wb}(A))$ is strongly positively invariant with respect to
  $ \subscr{f}{imcor}$. Note that the weight of any given edge is only
  modified by the agent who has it as an out-edge.  Consider a weight
  change done by the algorithm on an arbitrary edge $ (v_i,v_j) $, $
  v_i,v_j \in V $, and let us see the effect it has on the imbalances
  of the two vertices it affects. According to $ \subscr{f}{imcor} $,
  $ v_i $ increases the weight on her out-edge to $ v_j $ by $
  \epsilon \in \mathbb{R}_{>0} $ in order to balance
  herself. Consequently, the imbalance of agent $ v_j $ increases by
  at most $ \epsilon $.  Considered together, the weight change
  performed in the edge decrease the imbalance of $v_i $ by $\epsilon$
  while increasing the imbalance of $v_j$ by at most $\epsilon $; thus
  the value of that $\subscr{V}{wb}$ does not increase.  Since the
  edge $(v_i,v_j) $ is arbitrary, from~\eqref{eq:V}, this argument
  implies that $\subscr{V}{wb}$ is non-increasing along $
  \subscr{f}{imcor}$.

  Next, let us show that all evolutions of $
  (\Adj(G),\subscr{f}{imcor}) $ with initial condition in $
  \subscr{V}{wb}^{-1}(\le \subscr{V}{wb}(A)) $ are bounded.  Suppose
  otherwise, and take an initial condition that gives rise to an
  unbounded trajectory.  This implies that there exists, infinitely
  often, an agent $ v \in V$ with some positive imbalance. Let us
  justify that the infimum $\theta$ of all these positive imbalances
  is positive too. Since, by definition of $\subscr{f}{imcor} $, any
  imbalance in the timesteps after the initial one is a linear
  combination of the initial imbalances with coefficients $ 0 $ or $ 1
  $, the following inequality holds
  \[
  \theta\geq \nu = \min_{k\in\{1,\ldots,n\}} \min_{(i_1,\dots,i_k) \in
    C(n,k)} \{ |\omega_0(v_{i_1})+\omega_0(v_{i_2})
  +\ldots+\omega_0(v_{i_k})|\neq 0\},
  \]
  where $C(n,k)$ denotes the set of combinations of $k$ elements out
  of $n$, and $ \omega_0(v) $ refers to the initial imbalance of $ v
  \in V $.  Note that $\nu>0$.  Since $ \sum_{v\in V}\omega(v)=0 $,
  there exists a set of agents $ Y=\{v_1,\ldots, v_k\} \subset V $, $
  k \in \mathbb{Z}_{>0} $, with negative imbalance, such that $
  \sum_{v_i\in Y}\omega(v_i) $ is at most $ -\theta $ infinitely
  often. We show that this actually must happen always.  Suppose that
  at some point in the execution of the algorithm $\sum_{v_i\in
    Y}\omega(v_i)>-\theta $. After that, it will never happen that $
  \sum_{v_i\in Y}\omega(v_i)=-\theta $, otherwise, at least one of the
  agents has decreased her imbalance from a negative value, which is
  not possible by the definition of $ \subscr{f}{imcor} $.  But this
  contradicts $ \sum_{v_i\in Y}\omega(v_i) $ being at most $ -\theta $
  infinitely often. The above argument shows that there exists a
  positive imbalance of at least $ \theta $ that does not visit, even
  once, any edge going into the set $ Y $. Finally, since $G$ is
  strongly connected, this contradicts Lemma~\ref{lemma:min_weight}.

  Finally, recall that by Lemma~\ref{lemma:f_closed}, the map $
  \subscr{f}{imcor} $ is closed on $ \subscr{V}{wb}^{-1}(\le
  \subscr{V}{wb}(A))$.  With all the above hypotheses satisfied, the
  application of the LaSalle Invariance Principle for set-valued
  dynamical systems, cf. Theorem~\ref{theorem:LaSalle}, implies that
  the algorithm evolution approaches a set of the form $
  \subscr{V}{wb}^{-1}(c)\cap S $, where $ c \in \mathbb{R} $ and $ S $
  is the largest weakly positively invariant set contained in $ \{B\in
  \subscr{V}{wb}^{-1}(\le \subscr{V}{wb}(A)) \, | \, \text{there
    exists } \ C\in \subscr{f}{imcor}(B) \text{ such that }
  \subscr{V}{wb}(C)=\subscr{V}{wb}(B)\}$. In order to complete the
  proof, we need to show that $ c=0 $. We reason by
  contradiction. Suppose $ c>0$ and let $ A \in S \cap
  \subscr{V}{wb}^{-1}(c) $. Since $S$ is weakly positively invariant,
  there exists an evolution starting from $A$ which is contained in $
  S $, and hence stays in the level set $ \subscr{V}{wb}^{-1}(c) $.
  Let us show that after a finite time, this evolution will leave
  $\subscr{V}{wb}^{-1}(c) $, hence reaching a contradiction.  Since $
  c \in \mathbb{R}_{>0} $, there exists at least one agent $ v_i $
  with positive imbalance $ \omega(v_i)>0 $. Therefore, since $
  \sum_{v\in V}\omega(v)=0 $, there exists a set of agents whose
  imbalances are negative and sum at least $ -\omega(v_i) $.  Since
  the digraph is strongly connected, by Lemma~\ref{lemma:min_weight},
  after a finite time the positive imbalance $ \omega(v_i) $ will
  visit this set, which would strictly decrease the value of
  $\subscr{V}{wb}$, reaching a contradiction.  Finally, the
  finite-time convergence to the set of weight-balanced adjacency
  matrices follows from noting that the value of $\subscr{V}{wb}$ is
  initially finite and, by Lemma~\ref{lemma:min_weight}, there exists
  a finite number of iterations after which $\subscr{V}{wb}$ is
  guaranteed to have decreased at least an amount $2 \nu$.  The
  convergence to a weight-balanced adjacency matrix is a consequence
  of the finite-time convergence to the set.
\end{proof}


\begin{remark}\longthmtitle{Time complexity of the \wbda}
  {\rm Roughly speaking, the time complexity of an algorithm
    corresponds to the number of rounds required in order to achieve
    its objective. More formal definitions can be found
    in~\cite{FB-JC-SM:08cor,NAL:97,DP:00, GT:01}.  The
    characterization of the time complexity of the \wbda is an open
    problem, however, we characterize in
    Section~\ref{sec:weight-balanced-modified} the time complexity of
    a modified version of this algorithm.  \oprocend }
\end{remark}

\begin{example}\label{ex:dis_digraph_ex}\longthmtitle{Execution of the
    \wbda} {\rm Consider the digraph of
    Figure~\ref{fig:ex1}(a). The \wbda converges to a weight-balanced
    digraph in $ 6 $ rounds, as illustrated in
    Figure~\ref{fig:dis_digraph_ex_ev}.}
    \oprocend
    \begin{figure}
      \[
      (1) \ \ \xymatrix{ \bullet \ar[dr]^1 \ar[ddr]^1 & & \bullet
        \ar@{-->}[ll]^{\textbf{2}} \\ 
      & \bullet \ar@{-->}[ur]^{\textbf{3}} &\\
      \bullet \ar[ur]^1 & \bullet \ar[u]_1 \ar[uur]^1 \ar[l]^1 & } 
      \ \ (2) \ \ 
    \xymatrix{ \bullet \ar[dr]^1 \ar[ddr]^1 & & \bullet \ar@{-->}[ll]^{\textbf{4}} \\
      & \bullet \ar[ur]^{\textbf{3}} &\\
      \bullet \ar[ur]^1 & \bullet \ar[u]^1 \ar[uur]^1 \ar[l]^1 & }
      \ \ (3) \ \ 
    \xymatrix{ \bullet \ar@{-->}[dr]^{\textbf{3}} \ar[ddr]^1 & & \bullet \ar[ll]_{\textbf{4}} \\
      & \bullet \ar[ur]^{\textbf{3}} &\\
      \bullet \ar[ur]^1 & \bullet \ar[u]^1 \ar[uur]^1 \ar[l]^1 & }
    \]
    \[
    (4) \ \
    \xymatrix{ \bullet \ar[dr]^{\textbf{3}} \ar[ddr]^1 & & \bullet \ar[ll]_{\textbf{4}} \\
      & \bullet \ar@{-->}[ur]^{\textbf{5}} &\\
      \bullet \ar[ur]^1 & \bullet \ar[u]^1 \ar[uur]^1 \ar[l]_1 & }
      \ \ (5) \ \ \xymatrix{ \bullet \ar[dr]^{\textbf{3}} \ar[ddr]^1 & & \bullet \ar@{-->}[ll]_{\textbf{6}} \\
      & \bullet \ar[ur]^{\textbf{5}} &\\
      \bullet \ar[ur]^1 & \bullet \ar[u]^1 \ar[uur]^1 \ar[l]^1 & } 
      \ \ (6) \ \
    \xymatrix{ \bullet \ar[dr]^{\textbf{3}} \ar@{-->}[ddr]_{\textbf{3}} & & \bullet \ar[ll]_{\textbf{6}} \\
      & \bullet \ar[ur]^{\textbf{5}} &\\
      \bullet \ar[ur]^1 & \bullet \ar[u]^1 \ar[uur]^1 \ar[l]_1 & }
    \]
    \caption{Execution of the \wbda for the digraph of
      Figure~\ref{fig:ex1}(a). The initial condition is the adjacency
      matrix where all edges are assigned weight $1$. In each round,
      the edges which are used for sending messages are shown
      dashed. One can observe that the Lyapunov function $
      \subscr{V}{wb}(A_0)$ takes the following values in subsequent
      time steps: $6$, $4$, $4$, $4$, $4$, $4$, $0$.}
    \label{fig:dis_digraph_ex_ev}
  \end{figure}
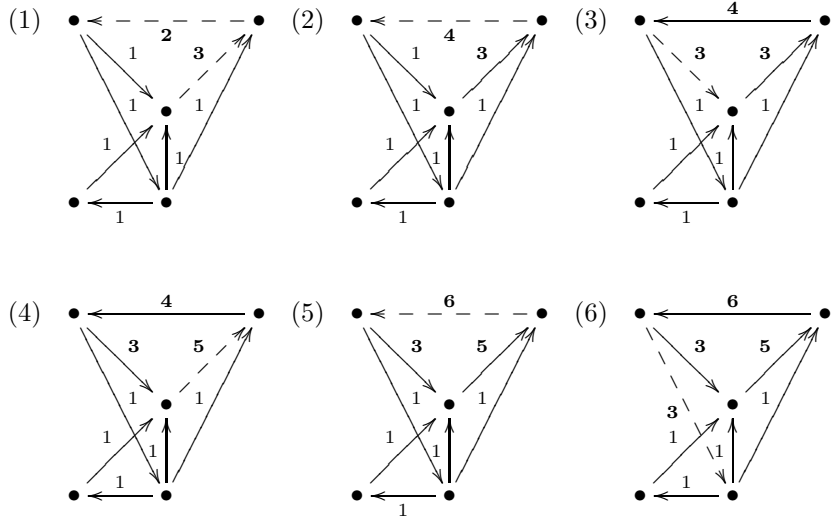
\end{example}

\subsection{The \wbmda}\label{sec:weight-balanced-modified}

In this section, we modify the
\wbda to synthesize a strategy, distributed over the mirror of the
digraph, which converges quickly to a weight-balanced digraph. This
procedure can be employed to construct a weight-balanced digraph
without computing any cycle. We start by an informal description of
the \wbmda. Let $ G $ be a strongly connected digraph of order $ n\in
\mathbb{Z}_{>0} $:
\begin{enumerate}
\item at each round, agents are able to receive messages from their
  in-neighbors. There is an initial round when each
  agent assigns a weight to each out-edge and sends it to her
  corresponding out-neighbor. In this way, everybody can compute her
  in-degree. After this, at most, only one out-edge per agent is
  changed in each round;
\item each agent with positive imbalance adds it to the weight on one
  of out-edges to one of the out-neighbors with minimum imbalance and
  sends a message to the corresponding out-neighbor informing her of
  the change.
\item \emph{multiple-messages rule}: if an agent receives more than
  one message from her in-neighbors, she adds the received messages to
  her imbalance;
\item \emph{fair-decision rule}: if the agent has more than one
  out-neighbor with the exact same minimum imbalance, it randomly
  chooses one.  However, the next time it needs to choose between her
  out-neighbors with the same minimum imbalance, she will choose a new
  out-neighbor.
\end{enumerate}

In the following, we give a formal definition of this algorithm. 
For all $ i\in\{1,\ldots, n\} $, let
\[
\subscr{\Omega}{min}(v_i)=\{v_j \in \upscr{\mathcal{N}}{out}(v_i) \ | \
\omega(v_j)=\min_{v_l\in \upscr{\mathcal{N}}{out}(v_i)} \omega(v_l)\}.
\]
We define $ \subscr{g}{imcor}: \Adj(G) \rightrightarrows \Adj(G) $
by
\begin{align}\label{eq:g}
  \subscr{g}{imcor}(A)&= \Big\{ B\in \Adj(G) \ | \ \text{for each} \;
  i \in \{1,\ldots, n\}, \text{there exists } j^* \; \mathrm{with} \;
  v_{j^*} \in \subscr{\Omega}{min}(v_i) \nonumber
  \\
  & \qquad \text{such that } b_{ij}=
  \begin{cases}
    a_{ij}+\omega(v_i), & \mathrm{if} \; \omega(v_i)>0 \;
    \text{and} \; j =j^*
    \\
    a_{ij}, & \mathrm{otherwise}
  \end{cases}
  \Big\}.
\end{align}
Note that the evolutions of the \wbmda are a subset of all the
evolutions of the set-valued map $ \subscr{g}{imcor}$ (because of the
fair-decision rule). The next result shows that this algorithm
converges in finite time to a weight-balanced digraph. 

\begin{theorem}[Convergence of the
  \wbmda]\label{theorem:modified_distributed}
  For a strongly connected digraph $ G $, any
  evolution under the \wbmda converges in finite time to a
  weight-balanced adjacency matrix.
\end{theorem}
\begin{proof}
  For an initial condition $A \in \Adj (G)$, consider the sublevel set
  $\subscr{V}{wb}^{-1}(\le \subscr{V}{wb}(A))$.  Similarly to the
  proof of Theorem~\ref{theorem:main}, observe that the continuous
  function $ \subscr{V}{wb} $ is non-increasing along $
  \subscr{g}{imcor} $ and thus $ \subscr{V}{wb}^{-1}(\le
  \subscr{V}{wb}(A)) $ is strongly positively invariant with respect
  to $ \subscr{g}{imcor} $.
An argument similar to the one in Lemma~\ref{lemma:f_closed} shows that 
   the map $ \subscr{g}{imcor} $ is closed at 
  $ \subscr{V}{wb}^{-1}(\le \subscr{V}{wb}(A))$.
  Now, pick an evolution of $ (\Adj(G),\subscr{g}{imcor}) $ with
  initial condition in $ \subscr{V}{wb}^{-1}(\le \subscr{V}{wb}(A))$
  that satisfies the fair-decision rule.  Then a similar version of
  Lemma~\ref{lemma:min_weight} can be used to establish that this
  evolution is bounded.  Thus by the LaSalle Invariance Principle for
  set-valued dynamical systems, Theorem~\ref{theorem:LaSalle}, this
  evolution approaches a set of the form $ \subscr{V}{wb}^{-1}(c)\cap
  S $, where $ c \in \mathbb{R}_{\geq 0} $ and $ S $ is the largest
  weakly positively invariant set contained in $ \{B\in
  \subscr{V}{wb}^{-1}(\le \subscr{V}{wb}(A)) \, | \, \text{there
    exists } \ C\in \subscr{g}{imcor}(B) \text{ such that }
  \subscr{V}{wb}(C)=\subscr{V}{wb}(B)\}$.  The fact that $c$ must be
  zero follows from the observation that, under the \wbmda, the value
  of $\subscr{V}{wb}$ decreases by a finite amount, bounded away by a
  positive constant, after a finite number of iterations.  The
  convergence in finite time can also be justified in a similar way to
  Theorem~\ref{theorem:main}.
\end{proof}

\begin{example}[Execution of the \wbmda]\label{example:mirrow} {\rm
    Figure~\ref{fig:example_mirror} shows a strongly connected
    digraph.
    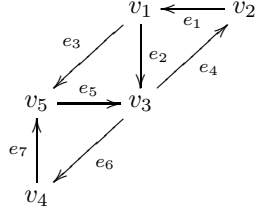
\begin{figure}
      \[
      \xymatrix{& v_1 \ar[d]^{e_2} \ar[dl]_{e_3} & v_2 \ar[l]^{e_1}\\
        v_5 \ar[r]^{e_5} & v_3 \ar[dl]^{e_6} \ar[ur]_{e_4} \\
        v_4 \ar[u]^{e_7} &} 	  
      \]
      \caption{A strongly connected digraph.}
      \label{fig:example_mirror}
    \end{figure}
    Figure~\ref{fig:mod_ex_1} shows an execution of the \wbmda for
    this digraph that converges in three rounds to a weight-balanced
    digraph.
    \begin{figure}[htb]
      \[
      \xymatrix{ & -1 \ar[d]_1 \ar[dl]_1 & 0 \ar[l]_1 & \\
        +1 \ar[r]^1 & 0 \ar[dl]^1 \ar[ur]^1 &\\
        0 \ar[u]^1 & & }\quad 
      \xymatrix{ & -1 \ar[d]_1 \ar[dl]_1 & 0 \ar[l]_1 & \\
        0 \ar@{-->}[r]_{\textbf{2}} & +1 \ar[dl]^1 \ar[ur]^1 &\\
        0 \ar[u]^1 & & }\quad 
      \xymatrix{ & -1 \ar[d]_1 \ar[dl]_1 & +1 \ar[l]_1 & \\
        0 \ar[r]_{\textbf{2}} & 0 \ar[dl]^1 \ar@{-->}[ur]^{\textbf{2}} &\\
        0 \ar[u]^1 & & }       
      \]
      \begin{center}
        \[
        \xymatrix{ & v_2
          \ar[d]_1 \ar[dl]_1 & v_1 \ar@{-->}[l]_{\textbf{2}} & \\
          v_3 \ar@{-->}[r]_{\textbf{2}} & v_4 \ar[dl]^1
          \ar@{-->}[ur]^{\textbf{2}} &\\
          v_5 \ar[u]^1 & & }
        \]
      \end{center}
      \caption{An execution of \wbmda for the digraph of
        Figure~\ref{fig:example_mirror}. The dashed lines show
        the edges which have been used to send messages. }
     \label{fig:mod_ex_1}
   \end{figure}
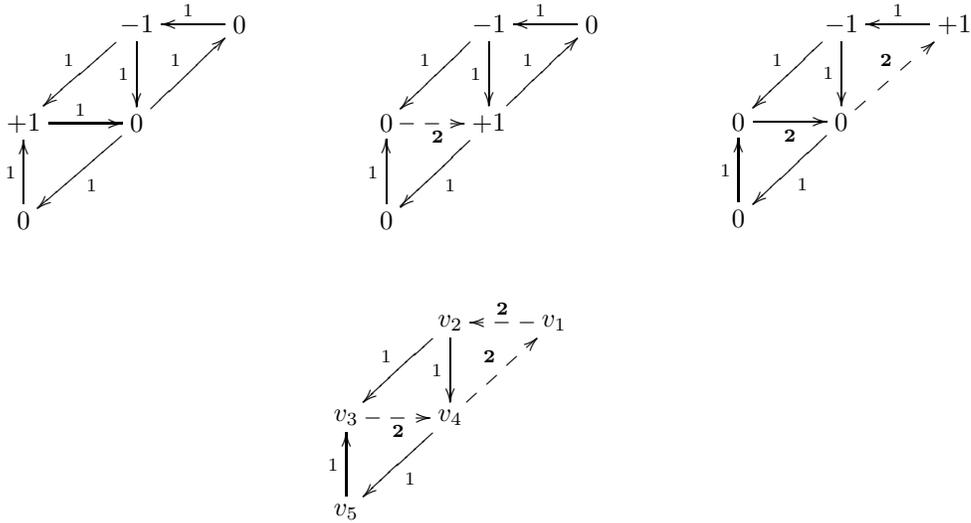
   Note that agent $ v_4 $ has two out-neighbors with the same
   imbalance, namely $ v_1 $ and $ v_5$.  If $v_4$ decides to update
   the weight on the edge $(v_4,v_5)$ to correct its imbalance, then
   the fair-decision rule will make her choose the edge $(v_4,v_1)$
   the next time.  Figure~\ref{fig:mod_ex_2} shows the result of the
   execution of the \wbmda in such a case.
   \begin{figure}[htb]
     \[
     \xymatrix{ & v_2 \ar[d]_1\ar[dl]_1 & v_1 \ar@{-->}[l]_{\textbf{2}} & \\
       v_3 \ar@{-->}[r]_{\textbf{3}} & v_4 \ar@{-->}[dl]^{\textbf{2}}
       \ar@{-->}[ur]^{\textbf{2}}&\\ 
       v_5 \ar@{-->}[u]^{\textbf{2}} & & }
     \]
     \caption{Another possible execution of \wbmda for
       Example~\ref{example:mirrow}.  The dashed lines show
       the edges which have been used to send messages. }
     \label{fig:mod_ex_2}
   \end{figure}
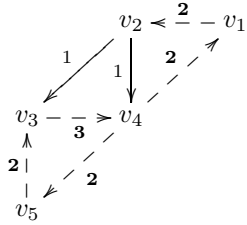
   If after the $4$th round of the execution the agent $ v_4 $ would
   keep updating the weight on the edge to agent $ v_5 $, then the
   algorithm would never converge to a weight-balanced
   digraph. \oprocend }
\end{example}

Next, we investigate the rate of convergence of the \wbmda. 

\begin{proposition}\longthmtitle{Time complexity of the
    \wbmda}\label{proposition:conv_rate_modified_dis}
  The time complexity of the \wbmda is in $ O(n^4) $.
\end{proposition}
\begin{proof} 
  Let $ G=(V,E) $ be a strongly connected digraph and consider the
  agent $ v_i\in V $ which initially has the maximum positive
  imbalance $ \omega(v_i)=r \in \mathbb{R}_{>0} $.  In the worst-case
  scenario, this imbalance should reach $ r $ agents with negative
  imbalance of $ -1 $ and it has to go through the longest path in
  order to reach the first agent with negative imbalance $ -1 $ and,
  after that, through the longest path to get to the second agent with
  negative imbalance and so on.  According to the execution of the
  \wbmda, agent $ v_i \in V $ updates the weight on the edge to an
  out-neighbor $ v_{i+1} \in V $. Then, agent $ v_{i+1} $ might pass
  the imbalance to the next agent in the longest path directly, or
  after sending it to a cycle which does not include any agent with
  negative imbalance.  Then, after some finite time, the agent $
  v_{i+1} $ will receive the positive imbalance $ r $ again and by the
  fair-decision rule, this time she will pick a different
  out-neighbor. Now, suppose that all the agents in this longest path
  have the maximum number of out-neighbors and furthermore, suppose
  that all the agents try all their other out-neighbors (via cycles)
  before finding the correct out-neighbor. Since the longest path, the
  maximum length of a cycle, and the maximum out-degree are all in $
  O(n) $, the time complexity of the positive imbalance $ r $ reaching
  the first agent with negative imbalance $-1$ is in $ O(n^3) $.
  Since the imbalance $ r $ is in $ O(n) $, the time complexity of
  \wbmda is in $ O(n^4) $.
\end{proof}

The result stated in
Proposition~\ref{proposition:conv_rate_modified_dis} is to be
contrasted with the time complexity of the centralized algorithm
proposed in~\cite{LHT:70} to construct a weight-balanced digraph,
which can be shown to be in $O(2^{n^2})$.

\section{Strategies for making a digraph doubly
  stochastic}\label{section:ds_dist}

In this section, we investigate the design of cooperative strategies
running on strongly connected communication networks that allow agents
to determine a set of edge weights that make the digraph doubly
stochastic.

\subsection{The \dslmo}\label{subsection:dis_ds_self-loop}
From Corollary~\ref{corollary:3_adding_loop}, we know that all
strongly connected digraphs can be made doubly stochastic by allowing
the addition of self-loops to the digraph.  Therefore, combining this
observation with the \wbda, one can find a distributed strategy that
allows the agents to make any strongly connected communication network
doubly stochastic, as we state next.

\begin{theorem}\longthmtitle{\dslmo}
  Let $ G$ be a strongly connected digraph. If the agents
  \begin{enumerate}
  \item execute the \wbda,
  \item compute the maximum out-degree,
  \item if necessary, add a self-loop with appropriate weight to make
    their out-degrees equal to the maximum out-degree, and finally,
  \item divide the weights on each out-edge by the max out-degree,
  \end{enumerate}
  then the resulting weighted digraph is doubly stochastic.
\end{theorem}

Note that there a number of distributed ways to perform (ii), see
e.g.,~\cite{NAL:97,DP:00}.

\begin{example}\longthmtitle{Execution of the
    \dslmo}\label{ex:self_loop_addition} 
  {\rm Consider the digraph of Figure~\ref{fig:ex1}(a).
    As we showed, there exists no doubly stochastic adjacency matrix
    associated to this digraph. However, if we allow the agents to
    modify the digraph by adding self-loops, a doubly stochastic
    adjacency matrix can be associated to this digraph. Using the
    \wbda, the agents can compute the weight-balanced adjacency matrix
    \[
    A=\left(\begin{array}{ccccc}
        0 & 6 & 0 &0 & 0\\
        0 & 0 & 3 & 3 & 0\\
        5 & 0 & 0 & 0 & 0\\
        1 & 0 & 1 & 0 & 1\\
        0 & 0 & 1 & 0 &0
      \end{array}\right).
    \]
    Now, if the agents execute the modification of the digraph by
    adding self-loops with appropriate weight, they compute the
    adjacency matrix
    \[
    A=\left(\begin{array}{ccccc}
        0 & 6 & 0 &0 & 0\\
        0 & 0 & 3 & 3 & 0\\
        5 & 0 & 1 & 0 & 0\\
        1 & 0 & 1 & 3 & 1\\
        0 & 0 & 1 & 0 & 5
      \end{array}\right).
    \]
    Finally, by executing a
    division on the weight of the out-edges, the agents compute the
    doubly stochastic adjacency matrix
    \[
    A=\left(\begin{array}{ccccc}
        0 & 1 & 0 &0 & 0\\
        0 & 0 & 1/2 & 1/2 & 0\\
        5/6 & 0 & 1/6 & 0 & 0\\
        1/6 & 0 & 1/6 & 1/2 & 1/6\\
        0 & 0 & 1/6 & 0 & 5/6
      \end{array}\right).
    \eqoprocend
    \]  
  }
\end{example}

\subsection{The \cregular}

Given a digraph $ G=(V,E)$ and $C \in \integerspositive$, in this
section we introduce a strategy, distributed over the mirror digraph,
that, upon completion, allows the group of nodes to declare whether
the graph $G$ is $ C $-regular. If it is, the strategy finds a set of
weights that makes the graph $C$-regular.  Note that, once such
weights have been found, agents can easily determine a doubly
stochastic weight assignment by dividing all entries by $C$.

Our strategy is essentially an adaptation of the Goldberg-Tarjan's
algorithm~\cite{AVG-RET:88} for the problem under consideration. This
point will be clearer when we address the proof of its correctness in
Theorem~\ref{th:cregular}.  Let us start with a formal description of
the strategy.

(Initialization)
\begin{enumerate}
\item[(i)] each agent can send/receive messages to/from her in- and
  out-neighbors. There is an initial round when each agent $v_i \in V$
  assigns a unit weight, denoted $ a_{ij}$, to each out-edge
  $(v_i,v_j)$, and sends this information to her corresponding
  out-neighbor $v_j$. In this way, every agent can initially compute
  her in-degree;

\item[(ii)] each agent $v_i \in V$'s memory contains a \emph{load}
  vector $ (L_s(v_i),L_t(v_i))$, whose entries are termed
  \emph{source-} and \emph{target-load}, respectively, and a
  \emph{height} vector $(H_s(v_i),H_t(v_i)) $, whose entries are
  termed \emph{source-} and \emph{target-height}, respectively. These
  variables are initialized as,
  \begin{alignat*}{2}
    L_s(v_i) & = C-\dout(v_i) , \quad && L_t(v_i) = \din(v_i)-C ,
    \\
    H_s(v_i) & =2 , \quad && H_t(v_i)=1 ;
  \end{alignat*}

\item[(iii)] each agent $ v_i\in V $ sends the initial source-height $
  H_s(v_i) $ to her out-neighbors and the initial target-height $
  H_t(v_i) $ to her in-neighbors.  When agents receive this
  information from their in- and out-neighbors, they compute
  \begin{align*}
    \upscr{H}{max-in}_s(v_i)=\max_{v_j \in \Nin(v_i)} H_s(v_j) ;
  \end{align*}
\end{enumerate}

(Algorithm steps)
\begin{enumerate}

\item[(iv)] at each time step, if $ L_s(v_i)>0 $ and
  \begin{description}

  \item[(\emph{push forward}):] there exists an out-neighbor $ v_j \in
    \Nout(v_i) $, where $ a_{ij}< C $ and $ H_s(v_i) > H_t(v_j) $,
    then agent $ v_i $ sends the load $ L_s(v_i) $ to $ v_j $ and
    resets
    \begin{align*}
      a_{ij}:=a_{ij}+ L_s(v_i) \quad \text{and} \quad L_s(v_i):=0 ;
    \end{align*}

  \item[(\emph{declare digraph not $ C $-regular}):] there exists no
    out-neighbor $ v_j \in \Nout(v_i) $ such that $ a_{ij}< C $ and $
    H_s(v_i) > H_t(v_j) $, agent $ v_i $ announces that the digraph is
    not $ C $-regular and the algorithm terminates.

  \end{description}

\item[(v)] at each time step, if $ L_t(v_i)>0 $ and
  \begin{description}
  \item[(\emph{push backward}):] there exists an in-neighbor $ v_k \in
    \Nout(v_i) $, where $ a_{ki}> L_t(v_i)+1 $ and $ H_t(v_i) >
    H_s(v_k) $, then agent $ v_i $ sends the load $ L_t(v_i) $ to $
    v_k $ and resets
    \begin{align*}
      a_{ki}:=a_{ki}-L_t(v_i) \quad \text{and} \quad L_t(v_i):=0 ;
    \end{align*}
  \item[(\emph{increase target-height}):] there exists no
    in-neighbor $ v_k \in \Nout(v_i) $ such that $ a_{ki}> L_t(v_i)+1
    $ and $ H_t(v_i) > H_s(v_k) $, agent $v_i$ resets
    \begin{align*}
      H_t(v_i):=\upscr{H}{max-in}_s(v_i)+1 ,
    \end{align*}
    and sends a message to her in-neighbors informing them of her new
    target-height.
  \end{description}

\end{enumerate}

We refer to the algorithm described above as the \cregular. Note that
this strategy is distributed over the mirror digraph of $G$. The next
result characterizes its convergence properties.

\begin{theorem}\longthmtitle{The \cregular finds a $ C $-regular
    weight assignment iff the digraph is doubly
    stochasticable}\label{th:cregular}
  Let $ G=(V,E) $ be a digraph and $C \geq \max \{\max_{v \in V}
  \dout(v), \max_{v \in V} \din(v)\}$.  If the \cregular announces
  that $ G $ is not $ C $-regular and $C \geq |E|-|V|+1$, then the
  digraph is not doubly stochasticable.  If $ G $ is doubly
  stochasticable and $ C\geq \ds(G) $, then the \cregular converges to
  a $ C $-regular digraph. In both cases, the algorithm terminates in
  $ O(|V|^2\times |E|)$ steps.
\end{theorem}

\begin{proof}
  We start by showing that the problem of finding a $ C $-regular
  digraph can be reduced to a maximum flow problem with positive lower
  bounds on the edge
  loads~\cite{RKA-TLM-JBO:93,CB-ASR-AL:05}. Consider the bipartite
  digraph $ \tilde{G} = (\tilde{V},\tilde{E})$, where
  \begin{itemize}
  \item $ \tilde{V} $ contains two copies of $ V $, named $V_u$ and
    $V_w$, a source node $ s $, and a target node $ t $, i.e.,
    \[
    \tilde{V}= \{s \} \cup V_u \cup V_w \cup \{ t\} .
    \]
    In other words, the nodes $u_i \in V_u$ and $w_i \in V_w$
    correspond to the agent $v_i \in V$;
  \item $ \tilde{E} $ is defined as follows: there is no edge between
    vertices in $V_u $ and there is no edge between vertices in $ V_w
    $. For each edge $ (v_i,v_j) $ in $ E $, there exists an edge $(
    u_i,w_j) \in \tilde{E}$; there is an out edge between $ s $ to all
    $ u_i$ in $V_u $ and an out-edge from each $ w_i$ in $V_w $ to $ t
    $.
  \end{itemize}
  Next, let us define a maximum flow problem on $ \tilde{G} $. Let the
  capacity on each edge of the form $ (s,u_i)$ or $(w_j,t) \in
  \tilde{E} $, $ i,j\in \{1,\ldots, |V|\} $ be exactly $ C $.  Let the
  capacity on each edge $ (u_i,w_j) \in \tilde{E}$ be lower bounded by
  $ 1 $ and upper bounded by $C$. We refer to the weight of such edges
  as $ \tilde{a}_{ij} $, and therefore, $ 1 \leq \tilde{a}_{ij} \leq
  C$.  Consider the following problem: find the maximum flow that can
  be sent from the source $ s $ to the target $ t $.  It is not
  difficult to see that, by definition of $ C $-regularity, the
  digraph $ G $ is $ C $-regular if and only if the maximum flow of
  the problem just introduced is $C |V| $.

  Following~\cite{CB-ASR-AL:05, RKA-TLM-JBO:93}, the maximum flow
  problem with lower bounds described above can be transformed into a
  regular maximum flow problem as follows.  Define $
  \bar{a}_{ij}=\tilde{a}_{ij}-1 $. The bounds $ 1 \leq \tilde{a}_{ij}
  \leq C$ now become $ 0 \leq \bar{a}_{ij} < C $. Let the capacity on
  each edge of the form $(s,u_i) \in \tilde{E}$ be
  \begin{align*}
    C -\sum_{(u_i,w_j) \in \tilde{E}}1 = C - \dout(v_i) .
  \end{align*}
  Let the capacity on each edge of the form $(w_i,t) \in \tilde{E}$ be
  \begin{align*}
    -C +\sum_{(u_j,w_i) \in \tilde{E}}1 = -C + \din(v_i) .
  \end{align*}
  The proof now follows from noting that the execution of the
  \cregular, when transcribed to this maximum flow problem, exactly
  corresponds to the execution of the distributed version presented
  in~\cite{TLP-IL-MB-SHD:05} of the preflow-push algorithm of
  Goldberg-Tarjan algorithm~\cite{AVG-RET:88}.
  Note that, given our discussion in
  Section~\ref{section:doubly_stochastic_sec}, it is enough to execute
  \cregular for $ C \geq |E|-|V|+1$
  (cf. Lemma~\ref{lemma:DS_character_bound}) to determine if the
  digraph $G$ is indeed doubly stochasticable. The time complexity of
  the algorithm is a consequence of~\cite[Theorem 3.11]{AVG-RET:88}.
\end{proof}

It is worth mentioning that without the characterization of doubly
stochasticable digraphs, a negative answer from \cregular for a given
$C$ would be inconclusive to determine if the digraph is indeed doubly
stochasticable. At the same time, the results of
Section~\ref{section:doubly_stochastic_sec} imply that, if a digraph
is doubly stochasticable and the \cregular is executed with $C \geq
\ds(G)$ (which in particular is guaranteed if $C \geq |E|-|V|+1$),
then the strategy will find a set of appropriate weights.

\begin{figure}[hbt]
  \[
  \xymatrix{v_1 \ar[rr]  &&  v_2 \ar[dd] \ar[ddll] \\
    &&\\
    v_4 \ar[uu] \ar@/_/[rr] && v_3 \ar[ll] \ar[lluu] }
  \]
  \caption{A doubly stochasticable digraph with $
    \ds(G)=2$.}\label{fig:cregular_ex}
\end{figure}

\begin{example}\longthmtitle{Execution of the
    \cregular}\label{example:cregular_ex}
  {\rm Consider the digraph shown in Figure~\ref{fig:cregular_ex}.
  One can easily check that this digraph is doubly stochasticable and
  $ \ds(G)=2 $. Suppose the agents want to find a $C$-regular weight
  assignment and they run the \cregular for $C=3 \geq \ds(G) =2$.
    Figure~\ref{fig:cregular_ex_exe} shows the execution of the
    algorithm. To represent its evolution, we associate to each vertex
    $ v_i $, $ i\in \{1,\ldots, 4\} $, the source- and target-loads
    and a subindex with the source- and target-height.  The algorithm
    converges to a $ 3 $-regular digraph in $ 5 $ iterations.  }
  \oprocend

  \begin{figure}[htb]
    \[
    \xymatrix{ \ (1): \
      (2,-1)_{(2,1)} \ar[rr]^{1}  & &(1,-2)_{(2,1)} \ar[dd]^1 \ar[ddll]^{\ \ \ \ \ \  1} \\
      & &\\
      \ \ \ \ \ \ \ (1,-1)_{(2,1)} \ar[uu]^1 \ar@/_/[rr]_1 & &
      (1,-2)_{(2,1)} \ar[ll]_1 \ar[lluu]_{\ \ \ \ \ \ 1} } \qquad 
      \xymatrix{ (2):
      \
      (0,-1)_{(2,1)} \ar@{-->}[rr]^{3}  & &(0,0)_{(2,1)} \ar[dd]^1 \ar@{-->}[ddll]^{\ \ \ \ \ \  2} \\
      & &\\
      \ \ \ \ \ \ \ (0,1)_{(2,1)} \ar[uu]^1 \ar@/_/@{-->}[rr]_2 & &
      (0,0)_{(2,1)} \ar@{-->}[ll]_2 \ar[lluu]_{\ \ \ \ \ \ 1} }
    \]
    \[
    \xymatrix{
      (3): \ 
      (0,-1)_{(2,1)} \ar[rr]^{3}  & &(0,0)_{(2,1)} \ar[dd]^1 \ar[ddll]^{\ \ \ \ \ \  2} \\
      & &\\ 
      \ \ \ \ \ \ \ (0,1)_{(2,\textbf{3})} \ar[uu]^1 \ar@/_/[rr]_2 & & (0,0)_{(2,1)} \ar[ll]_2 \ar[lluu]_{\ \ \ \ \ \  1} }
    \qquad     
    \xymatrix{
      \ \ \ (4): \ 
      (0,-1)_{(2,1)} \ar[rr]^{3}  & &(0,0)_{(2,1)} \ar[dd]^1 \ar[ddll]^{\ \ \ \ \ \   2} \\
      & &\\ 
      \ \ \ \ \ \ \ (0,0)_{(2,3)} \ar[uu]^1 \ar@/_/[rr]_2 & & (1,0)_{(2,1)} \ar@{-->}[ll]_1 \ar[lluu]_{\ \ \ \ \ \  1} }
    \]
    \[
    \xymatrix{
      (5): \ 
      (0,0)_{(2,1)} \ar[rr]^{3}  & &(0,0)_{(2,1)} \ar[dd]^1 \ar[ddll]^{\ \ \ \ \ \   2} \\
      & &\\ 
      \ \ \ \ \ \ \ (0,0)_{(2,3)} \ar[uu]^1 \ar@/_/[rr]_2 & & (0,0)_{(2,1)} \ar[ll]_1 \ar@{-->}[lluu]_{\ \ \ \ \ \    2} }
    \]
    \caption{The execution of the \cregular for the digraph in
      Figure~\ref{fig:cregular_ex}. At each iteration of the
      algorithm, the pair $ (L_s,L_t)_{(H_s,H_t)} $ is shown for each
      vertex. At each iteration, dashed lines represent the edges used
      by the vertices to send loads. Observe that~(2),~(3),~(4),~(5),
      respectively, correspond to the operations push forward,
      increasing target height, push backward, and push forward.}
    \label{fig:cregular_ex_exe}
  \end{figure}
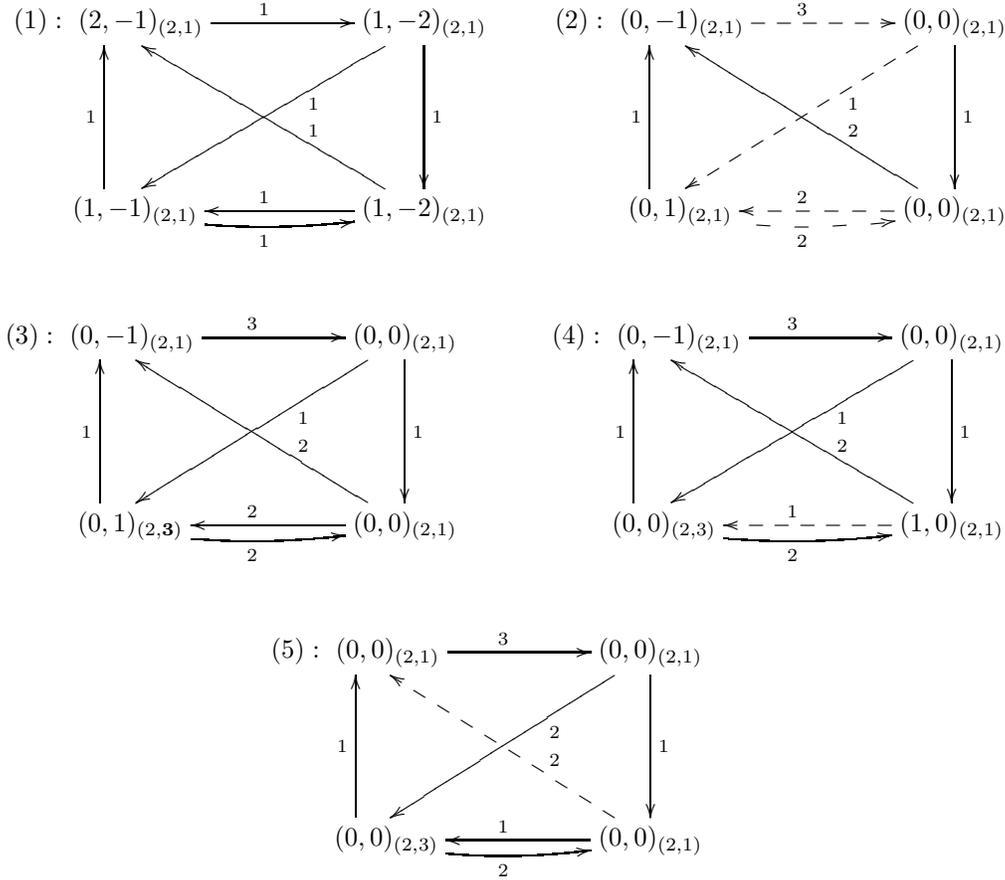
\end{example}

\section{Conclusions}\label{sec:conclusion_sec}

In this paper, we have fully characterized the properties of two
important classes of digraphs: weight-balanced and doubly stochastic
digraphs. Our results greatly enlarge the domain of systems for which
a variety of distributed formation and optimization algorithms can be
executed.  The first contribution is a necessary and sufficient
condition for doubly stochasticability of a digraph.  We have unveiled
the particular connection of the class of doubly stochasticable
digraphs with a special subset of weight-balanced digraphs. The second
contribution is the design of two discrete-time dynamical systems for
constructing a weight-balanced digraph from a strongly connected
digraph: (i) the \wbda, running synchronously on a group of agents,
and (ii) the \wbmda, distributed over the mirror of the original
digraph.  We have established the finite-time convergence of both
strategies via the set-valued discrete-time LaSalle Invariance
Principle.
We have also characterized the time complexity of the \wbmda, which is
substantially better than that of the existing centralized algorithm.
The third contribution is the design of two discrete-time dynamical
systems for constructing a doubly stochastic adjacency matrix for a
doubly stochasticable strongly connected digraph: (i) the \dslmo,
works under the assumption that agents are allowed to add weighted
self-loops. We showed that any strongly connected digraph can be
assigned a doubly stochastic adjacency matrix with this procedure;
(ii) the \cregular, distributed over the mirror digraph, for the case
when agents are not allowed to add self-loops.  The convergence of
this algorithm, and its time complexity, has been established by
formulating the problem as a constrained maximum flow problem.

Regarding future work, we would like to better understand the gap
between $\ds(G)$ and $p(G)$ for doubly stochasticable digraphs. We
would also like to study the case when zero edge weights are allowed
and, in particular, develop algorithmic procedures that can identify a
strongly connected spanning subdigraph which is doubly stochasticable.
To our knowledge, the synthesis of a distributed dynamical system that
computes doubly stochastic weight assignments when agents communicate
only over the original digraph, and not over the mirror digraph, is
still an open problem.  Finally, we would like to employ the results
of this paper to extend the range of applicability of existing
distributed algorithms for coordination, formation control, and
optimization tasks.



\end{document}